\documentclass[reqno]{amsart}

\allowdisplaybreaks[1]

\usepackage{graphicx, subfig}
\usepackage{amsfonts}
\usepackage{mathrsfs}
\usepackage{fancyhdr}
\usepackage{amsthm}
\usepackage{amssymb}
\usepackage{cases}
\usepackage{amsmath,amssymb,bm}
\usepackage{enumitem}
\usepackage{xcolor}
\usepackage{upgreek}
\usepackage{booktabs}
\usepackage{array}
\usepackage{float}
\usepackage{verbatim} 
\usepackage{caption}

\usepackage{tikz}
\usepackage{pgflibraryarrows}
\usepackage{pgflibrarysnakes}
\usepackage{appendix}
\usepackage{mathtools}

\providecolor{added}{rgb}{0,0,1}
\providecolor{deleted}{rgb}{1,0,0}
\usepackage[
bookmarks=true,  
bookmarksnumbered=true, 
colorlinks=true, pdfstartview=FitV, linkcolor=red, citecolor=blue,
urlcolor=blue]{hyperref}

\def\sfa{\mathsf{a}}

\def\sfb{\mathsf{b}}

\def\sfc{\mathsf{c}}

\def\cD{\mathcal{D}}
\def\bD{\mathbb{D}}
\def\EE{\mathbb{E}}

\def\sE{\mathscr{E}}
\def\cF{\mathcal{F}}
\def\cH{\mathcal{H}}

\def\sH{\mathscr{H}}

\def\cI{\mathcal{I}}

\def\cK{\mathcal{K}}

\def\cL{\mathcal{L}}
\def\sL{\mathscr{L}}
\def\bN{\mathbb{N}}

\def\cR{\mathcal{R}}
\def\RR{\mathbb{R}}
\def\R+{\mathbb{R}_+}

\def\sV{\mathscr{V}}

\def\upf{\textup{f}}
\def\uph{\textup{h}}
\def\upw{\textup{w}}

\def\TT{\mathbf{T}}

\def\ls{{\lesssim}}

\newcommand{\DD}{\mathbb D}

\renewcommand{\SS}{\mathcal S}

\def\RR{\mathbb{R}}
\def\EE{\mathbb{E}}

\def\ZZ{\mathbb{Z}}

\def\cD{{\mathcal D}}

\def\cF{{\mathcal F}}

\def\cH{{\mathcal H}}

\def\cP{{\mathcal P}}

\def\la{{\lambda}}
\def\si{{\sigma}}

\def\cP{{\mathcal  P}}

\def\cL{{\mathcal L}}
\def\Om{{\Omega}}

\def\al{{\alpha}}

\def\ga{{\gamma}}

\def\si{{\sigma}}

\def\la{{\lambda}}
\def\La{{\Lambda}}

\def\si{{\sigma}}
\def\al{{\alpha}}

\newcommand{\ep}{\varepsilon}

\newcommand{\1}{{\bf 1}}

\newcommand{\Blc}{\Big(}
\newcommand{\Brc}{\Big)}
\newcommand{\blk}{\big[}
\newcommand{\brk}{\big]}
\newcommand{\Blk}{\Big[}
\newcommand{\Brk}{\Big]}
\newcommand{\lc}{\left(}
\newcommand{\rc}{\right)}
\newcommand{\lk}{\left[}
\newcommand{\rk}{\right]}
\newcommand{\lt}{\left}
\newcommand{\rt}{\right}

\newcommand{\vertex}{\node[vertex]}

\newtheorem{thm}{Theorem}[section]

\newtheorem{prop}[thm]{Proposition}
\newtheorem{rmk}[thm]{Remark}

\newtheorem{defn}[thm]{Definition}
\theoremstyle{defn}

\newtheorem{exmp}[thm]{Example}

\theoremstyle{rmk}
\numberwithin{equation}{section}

\begin{document}

\title{Intermittency properties for a large class of stochastic PDEs driven by fractional  space-time  noises}

\author{Yaozhong Hu}
\address{Department of Mathematical and Statistical Sciences, University of Alberta, Edmonton, AB T6G 2G1, Canada}
\email{yaozhong@ualberta.ca}

\author{Xiong Wang}
\address{Department of Mathematical and Statistical Sciences, University of Alberta, Edmonton, AB T6G 2G1, Canada}
\email{xiongwang@ualberta.ca}


\date{}


\begin{abstract}
In this paper, we study intermittency properties    for various  stochastic PDEs with varieties of      space-time Gaussian  noises via matching upper and lower moment bounds of the solution. Due to the absence of the powerful Feynman-Kac formula, the lower moment bounds have been
 missing for many 
interesting equations except for the stochastic heat equation. This work introduces and explores
the Feynman diagram formula for the moments of the solution and the small ball nondegeneracy for the Green's function to obtain the lower bounds for all moments which match  the upper moment bounds. Our upper and lower moments are valid for various interesting equations, including stochastic heat equations, stochastic wave equations, stochastic heat equations with fractional Laplcians, and   stochastic diffusions which are both fractional in time and in space. 
\end{abstract}

\maketitle

\section{Introduction}
In this article, we consider the following   stochastic partial differential equation in the whole $d$-dimensional  Eulcidean space $\RR^d$: 
\begin{equation}\label{eq.LDiffu}
\sL u(t,x)=u(t,x)\dot{W}(t,x)\,,\qquad   t>0,~ x\in\RR^d\, 
\end{equation}
with some given initial condition(s).
Here $\sL$ denotes a general (including fractional order) partial differential operator and $\dot W(t,x)=\frac{\partial ^{d+1}}{\partial   t\partial x_1\cdots \partial x_d} W(t,x)$ is the Gaussian noise. Our approach can be applied to a large class of  operator  $\sL$. For this reason as     in \cite{HHNS15,NQ07}, instead of  giving  the  concrete form of $\sL$,  we shall impose conditions satisfied by  the Green's function associated with $\sL$.  

Let us briefly  recall the concept of Green's function. Suppose $f(t,x), t\ge 0,
x\in \RR^d$ is  a nice (smooth with compact support) function  and consider  the corresponding deterministic equation
\begin{equation}\label{eq.LDiffu_det}
\sL u(t,x)=f(t,x)\,,\qquad   t>0,~ x\in\RR^d\,.
\end{equation}
with the same initial condition(s) as 
in \eqref{eq.LDiffu}. The Green's function  associated with     $\sL$ is a (possibly generalized) function $G_{t,s}(x,y), 0\le s<t<\infty,
x,y\in \RR^d$ or a measure 
$G_{t,s}(x,y) dy:=G_{t,s}(x, dy)$ (we omit the explicit dependence of $G$ on $\sL$)  such that
the solution to \eqref{eq.LDiffu_det}  is given explicitly by
\begin{equation}\label{eq.MildSolL_det}
	u(t,x)=I_0(t,x)+\int_{0}^{t}\int_{\RR^d} G_{t,s}(x, y) f(s,y)dyds\,,
\end{equation}
where 
the term $I_0(t,x)$ depends on the initial data and the Green's function. 


If we formally replace $f(s,y)$ in \eqref{eq.MildSolL_det}   by $u(s,y)\dot W(s,y)$ and replace $\dot W(s,y)dsdy$
by   the Skorohod type stochastic integral $W(ds,dy)$,  then the solution to \eqref{eq.LDiffu} satisfies 
\begin{equation}\label{eq.MildSolL}
	u(t,x)=I_0(t,x)+\int_{0}^{t}\int_{\RR^d} G_{t,s}(x, y) u(s,y)W(d s,d y)\,,
\end{equation}
where the stochastic integral is interpreted in the Skorohod sense.
However, Unlike the previous identity  \eqref{eq.MildSolL_det}   the expression  \eqref{eq.MildSolL} is still an equation on $u$.  It is impossible  to make sense for each of the terms $\cL u(t,x)$ and
$u(t,x)\dot W(t,x)$ in a straightforward way
so it is impossible to find a solution satisfies \eqref{eq.LDiffu} literally. But it is possible to find $u(t,x)$ satisfies \eqref{eq.MildSolL}. A random field $u(t,x)$ satisfying \eqref{eq.MildSolL} will be called a mild solution (or random field solution)
to \eqref{eq.LDiffu}.  The existence and uniqueness of the   solution of 
\eqref{eq.MildSolL} have been well-studied together with the properties of the solution  for different operators $\sL$ and for different noise structures of $\dot W$.  For a recent survey on stochastic heat equation we refer to \cite{Hu19} and the references therein.  

One of the the most studied   properties of solution is the intermittency property
arose from the physics.  This property is related  the moment bounds of the solution. When 
\eqref{eq.LDiffu} is parabolic Anderson model, namely, when $\sL$ is a heat operator or fractional heat operator, then the  sharp (both lower and upper) moment  
 bounds are known, see 
\cite{BC2016, chen17, CHKN18,
CHS2018, CHSX2015, HHLN18, HHNT2015, lyu20},   and  we also refer to \cite{Hu19} and references therein. 
However,   when $\sL$ is wave operators 
(namely the hyperbolic Anderson model) or when $\sL$ is  (temporal) fractional differential operators,   the situation is different and as far as we know here are the progress achieved.
\begin{enumerate}
\item[(i)] Similar to   the stochastic heat equation (parabolic Anderson model)  by using the chaos expansion and the hypercontractivity inequality  one can obtain the (which we believe to be sharp) upper bounds.  It is also possible to obtain the  lower bound for the second moment.  There are many contributions to this topic and among them we mention only a few  
\cite{BC2014, BC2016, BJQ17, CHHH2017, SSX20}  and the references therein.   However, it is hard to obtain the sharp lower bound  for any $p$ moment  which matches  the upper bounds in terms of the growth of $p$. On the other hand, there are  works on the sharp second moment bound (ie $p=2$).  
\item[(ii)] Until now the success to obtain a sharp lower moment bounds largely relies on the clever application of the 
Feynman-Kac formula. However, there is no effective corresponding analogous formula for other equations. 
The only work  that we know is 
\cite[Theorem 4.1]{DM2009}, where the authors use an analogous  Feynman-Kac formula for stochastic wave equation obtained in \cite{DMT2008} to obtain a nice lower bound for all moments when the Gaussian noise is white in time and ``smooth" in space. Let us mention that after the completion of this work, we learned the announcement of a work \cite{Q21},  where the Gaussian noise is what they called 
 Dobri\'c-Ojeda one, namely, noise is  still white in time but with a  weight and the equation is one 
 dimensional stochastic wave equation.  The idea is still to make more careful use of the Feynman-Kac like 
 formula obtained in \cite{DMT2008}. 
\end{enumerate}
The objective of this paper is to obtain the sharp lower bounds for all   moments 
when the operator $\sL$ in \eqref{eq.LDiffu} is a wave operator or an  operator  which is  
fractional both in  time and in space.  The  Gaussian noise $\dot W $ can also be 
general. It does not need to be white or fractional in time.   

The approach that we use is   a generalization of the Feynman diagram formula. This formula allows us to keep track the terms in  the expectation of the product of several multiple Wiener-It\^o integrals which is very sophisticated.   
It is in some sense a
 brutal force method.  We fully explore the positivity of the Green's
 function. This property enables us to throw away some complicated terms and
 keep the main terms so that the remaining ones are possible to manage
 although still very sophisticated.  After 
the (fortunately   successful) 
isolation of the leading terms there
also remains an extremely  challenging problem 
of how to bound them from below. This is also a very serious issue. We find that many Green's 
functions satisfy what we call the small ball nondegeneracy property and this property can be used to ensure the sharp lower moment bounds.

 The Feynman diagram type formula that we obtain is essentially 
 analogous to the Feynman-Kac  type formula obtained in \cite{DMT2008}. However,
the former one  seems to be  more convenient for us to manage. Since we only use the 
positivity and the small ball nondegeneracy
properties, our approach
   is valid for a very large class of equations and for a large class of noise structure. 
 The equations include 
stochastic wave equation (SWE, $\sL=\partial_t^{2}-\Delta$),   stochastic  heat equation 
which is inhomogeneous and fractional in space
 ($(\alpha, A)$-SHE, $\sL=\partial_t-(-\nabla(A(x)\nabla))^{\alpha/2}$),  where $A$ is a positive definite
 symmetric matrix, 
 stochastic  partial differential  equations which is both fractional in time and in space 
 (SFDE, $\sL=\partial_t^{\beta}-\frac12(-\Delta)^{\alpha/2}$)
 (e.g. \cite{BC2014,BC2016,CHHH2017,CHS2018,CHW2018,DM2009} and references therein).   
 
We can also allow the  noise structure to be very general. We don't need it to be white or fractional 
in time or in  space. We follow the idea of \cite{HHNT2015} to assume that the covariance 
is bounded by some singular power functions. 

Here is the organization of the paper.  In Section \ref{s.2} we give the noise structure and
introduce the stochastic integral, mild solution, and  chaos expansion.   Section \ref{s.3} 
proposes  the general conditions  satisfied by the Green's function associated with 
$\sL$ and state our main results. 
Sections \ref{s.5} is devoted  to prove the   upper   moment bounds for  the solution.
This is done  by using the chaos expansion 
and the 
hypercontractivity inequality.  
Our main tool to prove the lower moment bounds is the generalization of the Feynman diagram 
formula for the expected value of product of several multiple Wiener-It\^o integrals. This formula is presented in Section \ref{s.4}.  After this preparation in Sections \ref{s.5} 
we  prove the   lower moment bounds of the solution.  To very that some famous operators $\sL$  satisfies the positivity and 
  small ball nondegeneracy conditions so that our results can be applied to cover a large class of interesting 
stochastic partial equations, we verify these  conditions for various   interesting  operators
$\sL$ in Section \ref{s.7}.

Throughout  the entire paper, we shall use the  notations    $\lesssim$,\ \ \  $\gtrsim$, \ and
$\simeq$  extensively.  The  meaning are 
conventional.  This is, $A \lesssim B$ (or $A \gtrsim B$) means that there are    constants $C \in (0, \infty)$  such that $A\le C B$ 
(or $B\le C A$, respectively). The notation   $A\simeq B$ means   that  both $A \lesssim B$ and $A \gtrsim B$ hold true.

\section{Noise covariance structure, mild solution and chaos expansion}\label{s.2} 
In this section, we   give the conditions satisfied by   the covariance of the noise $\dot W$ in \eqref{eq.LDiffu}. For this Gaussian noise we also define the (Skorohod type) integral, the mild solution,
and the chaos expansion of the solution candidate. These concepts are known, so we recall them very quickly to fix the notation throughout the paper. We refer to \cite{CHHH2017, H2017,   HHLNT17, HHLN18,  HHNT2015}
and the references therein for more details. 
The existence and uniqueness of the solution  in our new situation 
will be a consequence of the upper moment bounds.

\subsection{Noise covariance structure}
We assume that the noise   
 $\dot{W}(t,x)=\frac{\partial^{d+1}}{\partial t \partial {x_1} \cdots \partial {x_d}}$ $W(t,x)$ is mean zero Gaussian with the   following covariance  structure: 
\begin{equation} 
 \EE[\dot{W}(t,x)\dot{W}(s,y)]=\gamma(t-s)\Lambda(x-y)\,.
\end{equation}
The restriction that the covariance of the noise has this product form 
of a function of time variables  and a function of space variables 
is convenient. The reason that 
the time function is of the form $\gamma(t-s)$ means that the noise is stationary
(or the original process $W$ has stationary increment). The space function
is of the form $\Lambda(x-y)$ means that the noise is homogeneous. 

In order to simplify our presentation and in order to
cover the typical examples
we make the following  
assumptions.
For $\gamma$ we assume 
\begin{enumerate}[start=1,label=\textbf{(H\arabic*)}] 
	\item\label{H2}    There is a $\gamma\in (0, 1)$  such that     
	\[
	 c|t|^{-\gamma}\leq \gamma(t)\leq C|t|^{-\gamma}\,,\quad \forall \ t\in \RR_+  
	\]
for  some positive constants $c$, $C$. 
For convenience,  when  $\gamma=1$ we mean  $\gamma(t)=\delta(t)$.
\end{enumerate}
For  $\Lambda(\cdot) $ we assume that 
it  satisfies  one of the following three  conditions:
\begin{enumerate}[start=2,label=\textbf{(H\arabic*)}]
	\item\label{H3} There  is   $\lambda \in(0,d)$ such that
	\[
	 c|x|^{-\lambda} \leq \Lambda(x)\leq C|x|^{-\lambda}\,,\quad \forall \ x\in \RR^d\,.  
	\]
	\item\label{H4}   There  are constants  $\lambda_j\in(0,1), j=1, \cdots, d$ such that  
	\[
	 c\prod_{j=1}^d |x_j|^{-\lambda_j} \leq \Lambda(x)\leq C\prod_{j=1}^d |x_j|^{-\lambda_j}\,,\quad \forall \ x\in \RR^d\,. 
	\]	 
	In this case we denote $\la=\sum_{i=1}^d \la_i$.  
	\item\label{H5} When $d=1$ and $\gamma=1$,  we assume  $\Lambda(x)=\delta(x)$.
\end{enumerate}

\subsection{Stochastic integral}
We follow the approach of  
\cite{BC2016,H2017, HHNT2015, HN2009,N2006} to define stochastic integral.  First,  let us recall  the Fourier transform with respect to the spatial variables.  
Denote by  $\cD( \RR^{d})$   the space of real-valued infinitely differentiable functions with compact support on $\RR^d$ (We can also introduce $\cD(\RR_+\times  \RR^{d})$ in a similar way).  The Fourier transform 
is defined as  
\[
 \hat{f}(\xi)=\cF[f(\cdot)](\xi)=\int_{\RR^d}e^{-\iota \xi\cdot x}f(x)dx\,,
\]
and the the inverse Fourier transform is given by
\[
 \cF^{-1} f(x)=(2\pi)^{-d}\cF[f(\cdot)](-x)\,.
\]
Let $\cH$ is the Hilbert space defined as the completion of $\cD(\R+\times\RR^{d})$ equipped with the inner product given by
\begin{align}\label{Prod_H}
  \langle \phi,\psi \rangle_{\cH}&=\int_{(\R+\times\RR^{d})^2} \phi(t,x)\psi(s,y) \gamma(t-s)\Lambda(x-y)dtdxdsdy \\
  &=\frac{1}{(2\pi)^d} \int_{(\R+\times\RR^{d})^2} \gamma(t-s)\hat{\phi}(t,\cdot)(\xi)\overline{\hat{\psi}(s,\cdot)(\xi)} \mu(d\xi)\,,
\end{align}
where $\gamma:\RR\to\R+$ and $\Lambda:\RR^d\to \R+$ are non-negative definite functions and satisfy   \ref{H2} and one of the conditions \ref{H3}-\ref{H5}
introduced at the beginning of this section. 
Note that the space $\cH$ contains generalized functions.


The noise $\dot W$ can be described by an isonormal   family of mean zero Gaussian random variables
$\left\{ W(\phi)\,;  \phi\in \cD(\R+\times\RR^{d})\right\}$ with the covariance
$\EE[ W(\phi) W(\psi)]=\langle \phi,\psi \rangle_{\cH}$ for all  $\phi$ and $\psi$ in 
$\cD(\R+\times\RR^{d})$. This isometry  can be extended to   $\cH$ and  is denoted by
\[
 W(\phi)=\int_{\R+\times\RR^{d}} \phi(t,x)W(dt,dx)\,,\quad \hbox{for all $\phi\in\cH$}\,. 
\]
 
Let   $\cP$  be the set of     smooth and cylindrical
random variables of the form
\begin{equation*}
F=f(W(\phi_1),\dots,W(\phi_n))\,,
\end{equation*}
with $\phi_i \in \mathcal{H}$, $f \in C^{\infty}_p (\RR^n)$ (i.e.  $f$ and all
its partial derivatives have polynomial growth).  For $F\in \cP$ of the above form  we define  $DF$ as  the
$\mathcal{H}$-valued random variable  by the following expression
\begin{equation*}
DF=\sum_{j=1}^n\frac{\partial f}{\partial
x_j}(W(\phi_1),\dots,W(\phi_n))\phi_j\,.
\end{equation*}
The operator $D$ is closable from $L^2(\Omega)$ into $L^2(\Omega;
\mathcal{H})$  and we define the Sobolev space $\mathbb{D}^{1,2}$ as
the closure of  $\cP$ 
under the norm
\[
\|DF\|_{1,2}=\sqrt{\EE[F^2]+\EE[\|DF\|^2_{\mathcal{H}}]}\,.
\]
Given   any element $u \in L^2(\Omega;
\mathcal{H})$ if there is a $v\in L^2(\Omega)$ such that   
\begin{equation}\label{dual}
\EE  \lc v F \rc =\EE  \lc \langle DF,u
\rangle_{\mathcal{H}}\rc  \quad \hbox{ for any $F \in \mathbb{D}^{1,2}$ }
\end{equation}
 then we say that $u$ is in the domain 
  of $\delta$ and  we 
 call  it the {\it Skorohod integral} of $u$,
 denoted by
%

%
%
\[
 v=\delta(u)=\int_{0}^{\infty}\int_{\RR^d} u(t,x)W(d t,d x)\,.
\]
Obviously, when such $v$ (satisfying \eqref{dual}) 
exists, it is unique.
We refer to \cite{H2017, HHNT2015, N2006} for more details. 
Now with the Skorohod integral introduced, we give  the concept of mild solution as follows.
\begin{defn}
	An adapted random field $\{u(t,x):t\geq 0\,,x\in\RR^d\}$ 
	so that   $\EE[|u(t,x)|^2]<\infty$ for all $t\geq 0$ and $x\in\RR^d$ is called a mild solution to equation \eqref{eq.LDiffu} if for all $(t,x)\in \RR_+\times\RR^d$ the process
	\[
	 \{G _{t-s}(x,y)u(s,y)\1_{[0,t]}(s)\,:\,s\geq 0\,,y\in\RR^d\}
	\]
	is Skorohod integrable, and $u(t,x)$ satisfies
	\begin{equation}\label{eq.MildSolL'}
	u(t,x)=I_0(t,x)+\int_{0}^{t}\int_{\RR^d} G_{t-s}(x, y) u(s,y)W(d s,d y)\,,
\end{equation}
where $G_{t}(x,y)$ is the Green's function associated with $\sL$ and $I_0(t,x)$ is from  the initial
condition(s) and the Green's function.  
\end{defn}

%
%



If $u$ is a mild solution to \eqref{eq.LDiffu}, namely if 
$u$ satisfies \eqref{eq.MildSolL'},  then 
$u(s,y)=I_0(s,y)+\int_{0}^{t}\int_{\RR^d} G_{s-r}(y, z) u(r,z)W(d r,d z)$.
Substituting   this expression into \eqref{eq.MildSolL'}  
we  obtain 
\begin{eqnarray*}
u(t,x)
&=&I_0(t,x)+\int_{0}^{t}\int_{\RR^d} G_{t-s}(x, y) I_0(s,y)W(d s,d y)\\
&&\qquad +\int_{0}^{t}\int_0^r \int_{\RR^{2d}} G_{t-s}(x, y) 
	G_{s-r}(y, z)u(r,z) W(dr,dz) W(d s,d y) \,.  
\end{eqnarray*}
Repeating   this procedure
  we obtain a solution candidate for 
the  equation \eqref{eq.MildSolL'}:  
\begin{equation}\label{eq.WCE}
	u(t,x)=I_0(t,x) +\sum_{n=1}^{\infty} I _n(f_n(\cdot,t,x))\,.  
\end{equation}
Here 
\begin{align}\label{eq.f_n_symm}
& {f_n}(\cdot,t,x):=  {f_n}(t_1,x_1,\cdots,t_n,x_n,t,x)\\
=&\frac{1}{n!} \sum_{\sigma\in S_n} G_{t-t_{\sigma(n)}}(x,x_{\sigma(n)})G_{t_{{\sigma(n)}}-t_{{\sigma(n-1)}}}(x_{{\sigma(n)}},x_{{\sigma(n-1)}})	\cdots \nonumber\\
&\qquad\quad \times G_{t_{{\sigma(2)}}-t_{{\sigma(1)}}}(x_{{\sigma(2)}},x_{{\sigma(1)}})I  _0(t_{{\sigma(1)}},x_{{\sigma(1)}}) \1_{\{0<t_{{\sigma(1)}}<\cdots<t_{{\sigma(n)}}<t\}} \nonumber
\end{align}
is the symmetrization of 
\begin{align}\label{eq.f_n} 
 &G_{t-t_n}(x,x_n)G_{t_{n}-t_{n-1}}(x_{n},x_{n-1})\cdots 	\nonumber\\
& \times G_{t_{2}-t_{1}}(x_{2},x_{1})I _0(t_1,x_1) \1_{\{0<t_1<\cdots<t_n<t\}}  \,, 
\end{align}
where $S_n$ denotes the set of all permutations of $\{1,\dots,n\}$;     and $I_n(f_n(\cdot, t,x))$ is the multiple Wiener-It\^o  integral (e.g. \cite{H2017, N2006}).  
%
%
%
The expression \eqref{eq.WCE} is called the Wiener chaos expansion 
(or simply chaos expansion) of the solution.  It is known that if \eqref{eq.WCE} is convergent in $L^2(\Om)$, then
\eqref{eq.LDiffu} has a unique mild solution.

\section{Small ball nondegeneracy and main results} \label{s.3}  
\subsection{Small ball nondegeneracy for Green's function} 
Our main aim of this paper is to study the lower and upper asymptotics of the moments of mild solution defined in \eqref{eq.MildSolL}, which match with each other. 
What we need is the following assumptions on the Green's function associated with the operator $\sL$. The following assumptions is made in order to 
derive the sharp lower asymptotics:
\begin{enumerate}[label=\textbf{(G\arabic*)}]
	\item\label{A1}[\textbf{Positivity}]: $G_t (\cdot,\cdot)$ is a positive function, 
	measure, or generalized function.  
	\item\label{A2}[\textbf{Small ball nondegeneracy}]: $G_t (\cdot,\cdot)$ satisfies the \emph{small ball nondegeneracy (B($\sfa,\sfb$))}. This is, there exist  real numbers $\sfa$ and $\sfb$ (depending on the Green's function) satisfying 
	\begin{equation}\label{Cond.G2}
{ 		\sfa>-1\,,\quad \sfb>0\,,\quad \text{and}\quad \sfb(2\sfa+1)-\lambda>0\,,}
	\end{equation}
	and there is a constant $C>0$ such that
	\begin{equation}\label{A.HSB}
	\inf_{y\in   B_{\ep}(x)}	\int_{B_{\ep}(x)}G_{t} (y,z)dz \geq C \cdot t^\sfa\,,
	\end{equation}
	for all $0<t\leq \ep^\sfb \le 1$ and $x\in\RR^d$,  where $B_{\ep}(x)$ is the ball of center $x$ with radius $\ep$.  
\end{enumerate}
To  obtain  the  upper bound for moments 
what    we need is the following hypothesis for the Green's function. 
\begin{enumerate}[start=3,label=\textbf{(G\arabic*)}]
	\item\label{A3}[\textbf{HLS-type mass property}]: $G_t (\cdot,\cdot)$ satisfies what we shall call  the \emph{Hardy-Littelewood-Sobolev type  mass property} \emph{M$(\hbar)$}. That  is, there exist a real number $\hbar$ and a constant $C>0$ satisfying 
	 \begin{equation}\label{Cond.G3}
	 	\hbar >-1\,,
	 \end{equation}
	   and  
	 \begin{equation}\label{A.F_upper}
	 	\sup\limits_{x,x'\in\RR^d} \int_{\RR^{2d}} G_t(x,y)\Lambda(y-y') G_t(x',y')dydy' \leq C\cdot t^{\hbar}\,.
	 \end{equation}
\end{enumerate}
\begin{rmk}
	The task to verify the assumptions  
\textbf{(G1)-(G3)}, in particular to find the 
sharp indices $\sfa, \sfb, \hbar$ in \eqref{A.HSB}-\eqref{A.F_upper} is not trivial. We shall dedicate one section (Section \ref{s.7})
to verify  these conditions for various   partial differential operators $\cL$  that are     currently interested  by researchers. For  different operators $\sL$, we shall obtain the best indices $\sfa, \sfb, \hbar$ in the sense that our upper and lower $p$-moment  will match each other as $p$ or $t$ tends to infinity. 
\end{rmk}

\begin{rmk}
The hypothesis \ref{A3} is quite  standard for the upper moment bounds. When $G_t$ is a function (rather than a measure), then we can easily apply Hardy-Littelewood-Sobolev inequality (\cite[Theorem 4.3]{LL1997}) to obtain 
\begin{align*}
	\sup\limits_{x,x'\in\RR^d}& \int_{\RR^{2d}} G_t(x,y)\Lambda(y-y') G_t(x',y')dydy' \\
	&\leq \sup\limits_{x,x'\in\RR^d} \int_{\RR^{2d}} G_t(x,y)|y-y'|^{-\lambda} G_t(x',y')dydy' \\
	&\leq \sup\limits_{x\in\RR^d} \lk \int_{\RR^{d}}|G_t(x,y)|^{\frac{2d}{2d-\lambda}} dy\rk^{\frac{2d-\lambda}{d}}\,.
\end{align*}
 Then one can check the last integral is less than $C\cdot t^{\hbar}$ for some $\hbar>-1$ and $C>0$ to verify \eqref{A.F_upper}.
\end{rmk}
\begin{rmk}
In this remark, we give some intuitive connections between \ref{A2} and \ref{A3}. If we assume $G_t (x,y)=G_t(x-y)$ satisfies what we shall call  the \emph{total weighted  mass property} \emph{$\bar{M}(\upmu ,\upnu )$}: 
there exist  real numbers $\upmu$ and $\upnu$ (depending on the Green's function) satisfying 
	 \begin{equation}\label{Cond.G3_1}
	 	\upmu >-1\,,\quad \upnu \in\RR\,,\quad \text{and}\quad\upmu +\upnu >-1\,,
	 \end{equation}
	 and there are two positive constants $C_1$ and $C_2$ such that
	\begin{equation}\label{A.F_upper_1}
		\begin{cases}
		\int_{\RR^d} G_t(y)dy\leq C_1\cdot t^{\upmu }\,,\\
		\sup\limits_{x\in\RR^d} \int_{\RR^d} G_t(x-y)\Lambda(y)dy \leq C_2\cdot t^{\upnu }\,.
	\end{cases}	
	\end{equation}
Then we can easily see \eqref{A.F_upper} holds with $\hbar=\upmu+\upnu>-1$. Furthermore, we notice that  $\upmu=\sfa$,   where $\sfa$ is the same parameter in  \ref{A2}.

Let us discuss the SHE and SWE in one dimension as examples. It is easy to see   from Hardy-Littlewood rearrangement inequality  (see \cite[Theorem 3.4]{LL1997})  that 
	\begin{align*}
		\sup_{y\in\RR} \int_{\RR} G^{\uph}_t(x-y)\Lambda(x)dx \leq& \int_{\RR} \frac{1}{\sqrt{2\pi t}}e^{-\frac{x^2}{2t}}|x|^{-\lambda}dx \leq C\cdot t^{-\frac{\lambda}{2}}\,; \\
		\sup_{y\in\RR} \int_{\RR} G^{\upw}_t(x-y)\Lambda(x)dx \leq& \int_{-t}^{t} |x|^{-\lambda}dx \leq C\cdot t^{1-\lambda}\,,
	\end{align*}
	where $G^{\uph}_t$ and $G^{\upw}_t$ are heat kernel and wave kernel, respectively. In addition, we know that 
	\begin{equation*}
		\int_{\RR} G^{\uph}_t(x)dx=1\,,\quad \int_{\RR} G^{\upw}_t(x)dx=t\,.
	\end{equation*}
	Thus, \eqref{A.F_upper} holds  for SHE and SWE. Moreover, it will be  shown in Section \ref{s.7} that SHE and SWE satisfy \emph{small ball nondegeneracy} with \emph{$\sfa=0$} and \emph{$\sfa=1$}, respectively.
\end{rmk}

\subsection{Main results}  
In this subsection  we present our main results. This is, we give the upper moment estimates in Theorem \ref{thm.UMB} and lower moment in Theorem \ref{thm.LMB}. In fact, with $\gamma(\cdot)$, $\Lambda(\cdot)$ and $G_t(\cdot)$ satisfying conditions stated before, we also give the relation among the indices $\sfa$,
$\sfb$ and $\hbar$ so that the exponents in $t$ and $p$ in the lower and upper moments match with each others (see the Table  \ref{Table.LUM}).

First we state the result for the upper moment bounds. 
\begin{thm}\label{thm.UMB}
Assume 
	 $\gamma(\cdot)$ satisfies (the upper inequality in) \ref{H2} and $\Lambda(\cdot)$ satisfies (the upper inequality in) one of \ref{H3}-\ref{H5}.
	 Let  the Green function $G_t(\cdot)$ satisfy \ref{A3}.    Assume that the initial condition
	 term $I_0(t,x)$ is bounded, namely, 
	 there is a positive constant $C$ such that $\sup_{(t,x)\in 
	 \RR_+\times \in \RR^d}
	 |I_0(t,x)|\le C$.      Then there is a unique   mild solution $u(t,x)$   satisfying  \eqref{eq.MildSolL}. Moreover,   there are some constants $C_1$ and $C_2$ do not depend on $t$, $p$ and $x$ such that
\begin{align}\label{eq.UMB_thm}
		\EE\lk|u(t,x)|^p \rk \leq& \, C_1 \exp\lc C_2\cdot t^{1+\frac{1-\gamma}{\hbar +1}}\cdot p^{1+\frac{1}{\hbar +1}} \rc\,. 
	\end{align}
\end{thm}
The proof of this result will be given in next section  (Section \ref{s.4})  by using the hypercontractivity inequality.   

\begin{rmk} This result is   new in the sense that it holds now for general operator  $\sL$ satisfying 
\ref{A3}.  When $\sL$ is the heat operator,  wave operator,
fractional $\al$-diffusion operator, or partial differential operator both fractional in time and space but homogeneous in space, the result is known 
(e.g. \cite{BC2016, CHHH2017,  Hu19, SSX20}, references therein and   other references.
\end{rmk} 
Our main contribution of this work is the following lower moment bounds
for a general partial differential operator $\sL$. 
\begin{thm}\label{thm.LMB}
Assume  $\gamma(\cdot)$ satisfies (the lower inequality in)  \ref{H2} and $\Lambda(\cdot)$ satisfies (the lower inequality in)  one of \ref{H3}-\ref{H5}. 
Let the Green function $G_t(\cdot)$ satisfy (the lower inequality in)  \ref{A1} and \ref{A2}.   
If the initial condition satisfies $\inf_{(t,x)\in \RR_+\times \RR^d} I_0(t,x)\ge 
c_0 $ for some constant $c_0>0$,   then there are some positive 
constants $c_1$ and $c_2$ independent of    $t$, $p$ and $x$ such that we have
\begin{align}\label{eq.LMB_thm}
		\EE\lk|u(t,x)|^p \rk \geq \, c_1 \exp\lc c_2\cdot t^{1+\frac{\sfb\cdot(1-\gamma)}{\sfb (2\sfa+1)-\lambda}}\cdot p^{1+\frac{\sfb}{\sfb (2\sfa+1)-\lambda}} \rc\,.
	\end{align}
\end{thm}
The proof of this theorem replies on the Feynman diagram formula
for the moments of a chaos expansion. This formula will be presented in Section \ref{s.5}  and 
will be used in Section \ref{s.6} to prove the above theorem.

Consequently, combing Theorem \ref{thm.UMB} and \ref{thm.LMB} we obtain the following theorem about the matching  upper and lower moment bounds. 
\begin{thm}\label{thm.LUMB}
Assume  $\gamma(\cdot)$ satisfy \ref{H2} and $\Lambda(\cdot)$ satisfy  one of \ref{H3}-\ref{H5}.
Assume the Green function $G_t(\cdot)$ satisfy  \ref{A1}-\ref{A3} with
	\begin{equation}\label{Cond.G2_3}
		\hbar :=2\sfa-\frac{\lambda}{\sfb}>-1
	\end{equation}  
If the initial condition satisfies $ c_0\le  I_0(t,x) \le  
C_0 $ for some positive constants $0<c_0 < C_0<\infty$,   then  the mild solution $u(t,x)$ to \eqref{eq.LDiffu} satisfies 
	\begin{align}\label{eq.ULMB_thm}
	&	  c_1 \exp\lc c_2\cdot t^{1+\frac{\sfb\cdot(1-\gamma)}{\sfb (2\sfa+1)-\lambda}}\cdot p^{1+\frac{\sfb}{\sfb (2\sfa+1)-\lambda}} \rc
	\nonumber\\
	 &\qquad\qquad 
		  \le \EE\lk|u(t,x)|^p \rk \leq 
		  c_1 \exp\lc c_2\cdot t^{1+\frac{\sfb\cdot(1-\gamma)}{\sfb (2\sfa+1)-\lambda}}\cdot p^{1+\frac{\sfb}{\sfb (2\sfa+1)-\lambda}} \rc
	\end{align}
	for all $t\ge 0$, $x\in \RR^d$, $p\in \ZZ_+$,  where   $c_1, c_2, C_1, C_2$ are some  positive constants, independent of $t,x,p$.  
\end{thm}
\begin{proof} It is obvious that 
under the conditions of this theorem both the conditions of Theorems \ref{thm.UMB} and \ref{thm.LMB} hold. Thus both \eqref{eq.LMB_thm} and \eqref{eq.UMB_thm} hold true. Replacing 
$\hbar$ by \eqref{Cond.G2_3} we see that 
\eqref{eq.UMB_thm} becomes the second inequality
in \eqref{eq.ULMB_thm}. The theorem is then proved. 
\end{proof}


We shall demonstrate that 
\eqref{Cond.G2_3} holds true for 
the Green's function of various partial differential operators:  SHE, $\alpha$-SHE, SWE and SFD  (see \eqref{eq.SHE}, \eqref{eq.SHE_al}, \eqref{eq.SWE} and \eqref{eq.FDiffu} respectively). We summarize the results of that 
section here  in  following table. Notice that Table \ref{Table.LUM} only includes the 
exponent parts of \eqref{eq.ULMB_thm}. 

\renewcommand\arraystretch{2}
\begin{table}[H]
\centering
	\caption{Matching Lower and Upper Moments}
	\scalebox{0.9}{
	\begin{tabular}{lcccc}
	\toprule  
	SPDEs&  ($\sfa$,$\sfb$)&  $\hbar$  & Moment& When $\gamma=2-2H$\\
	\midrule  
	SHE& \emph{(0,2)}& $-\frac{\lambda}{2}$ & $t^{1+\frac{2\cdot(1-\gamma)}{2-\lambda}}\cdot p^{\frac{4-\lambda}{2-\lambda}}$ & $t^{\frac{4H-\lambda}{2-\lambda}}\cdot p^{\frac{4-\lambda}{2-\lambda}}$ \\
	$\alpha$-SHE& \emph{(0,$\alpha$)}& $-\frac{\lambda}{\alpha}$ & $t^{1+\frac{\alpha\cdot(1-\gamma)}{\alpha-\lambda}}\cdot p^{\frac{2\alpha-\lambda}{\alpha-\lambda}}$ & $t^{\frac{2H\alpha-\lambda}{\alpha-\lambda}}\cdot p^{\frac{2\alpha-\lambda}{\alpha-\lambda}}$ \\
	SWE& \emph{(1,1)} & $2-\lambda$  & $t^{1+\frac{1-\gamma}{3-\lambda}}\cdot p^{\frac{4-\lambda}{3-\lambda}}$ & $t^{\frac{2H+2-\lambda}{3-\lambda}}\cdot p^{\frac{4-\lambda}{3-\lambda}}$ \\
	SFD& \emph{($\beta-1$,$\frac{\alpha}{\beta}$)} & $2(\beta-1)-\frac{\lambda\beta}{\alpha}$ & $t^{1+\frac{\alpha(1-\gamma)}{2\alpha\beta-\alpha-\beta\lambda}}\cdot p^{\frac{\beta(2\alpha-\lambda)}{2\alpha\beta-\alpha-\beta\lambda}}$ & $t^{\frac{\alpha(2\beta+2H-2)-\beta\lambda}{2\alpha\beta-\alpha-\beta\lambda}}\cdot p^{\frac{\beta(2\alpha-\lambda)}{2\alpha\beta-\alpha-\beta\lambda}}$ \\
	\bottomrule 
	\label{Table.LUM}
	\end{tabular}}
\end{table}

\section{Upper moment  bounds}\label{s.4}
Our goal of this section is to prove the  upper moment  
bounds assuming that  \ref{A1}, \ref{A3}  hold for the Green's function $G$ 
associated with the operator $\sL$ and  assuming that \ref{H2} and one of 
\ref{H3}-\ref{H5} or one of \ref{H6}-\ref{H7} hold true for the noise covariance structure.

Sometimes it is convenient to use Fourier transformation  to represent the
covariance function in spatial variables.  Assume $\La(x)\ge 0$ for all $x\in \RR^d$ and assume that
there is a measure $\mu$ on $\RR^d$ such that 
\begin{equation}
\La(x)=\int_{\RR^d} e^{\iota x\xi} \mu(d\xi) ,.   
\end{equation}
We now assume some conditions on the Fourier mode that are similar to  \ref{H3}-\ref{H4} and \ref{A3}.  
 \begin{enumerate}[start=2,label=\textbf{(H\arabic*$^\prime$)}] 
	\item\label{H6}    There is a  $\hat \La:\RR^d\rightarrow \RR$ such that 
$\mu(d\xi)=\hat \La(\xi) d\xi$ and there  are constants  $\lambda_j\in(0,1), j=1, \cdots, d$
and $C>0$  such that 
	\[
	 |\hat \La(\xi)|\leq C
	 \prod_{j=1}^d|\xi|^{\la_j-1}\,,\quad \forall \ \xi\in \RR^d\,.
	\]  
	In this case we denote $\la=\la_1+\cdots+\la_d$. 
	\item\label{H7}   There is a  $\hat \La:\RR^d\rightarrow \RR$ such that 
$\mu(d\xi)=\hat \La(\xi) d\xi$ and there  are constants $\lambda \in (0,d)   $ and $C>0$  such that  
	\[
	 |\hat \La(\xi)|\leq C|\xi|^{\la -d}
	 \,,\quad \forall \ \xi\in \RR^d\,. 
	\] 
\end{enumerate}
\begin{enumerate}[start=3,label=\textbf{(G\arabic*$^\prime$)}]
	\item\label{A3'}[\textbf{Majorized property}]  $G_t(\cdot)$ satisfies the \emph{Majorized property ($M (\hbar)$)}. This is, there exists a positive function or measure $Q_t$ such that $G_t(x,y)\leq Q_t(x-y)$ for any $t>0$ and $x,y\in\RR$, and there exist a real number $\hbar>-1$ (the same ones in \ref{A3}) and a constant $C>0$ such that
	\begin{equation}\label{A.F_upper'}
		\sup_{\eta\in\RR^d} \int_{\RR^d}|\hat{Q}_t(\xi-\eta)|^2|\mu|(d\xi) \leq C \cdot t^{\hbar}\,,
	\end{equation}
	where $|\mu|(\xi)=|\hat \La(\xi)| d\xi$ with \ref{H6} or \ref{H7} holds.
\end{enumerate}

\begin{thm}\label{thm.UMB2}
Let  the Green function $G_t(\cdot)$ satisfy \ref{A3'}. Assume 
	 $\gamma(\cdot)$ satisfies (the upper inequality in) \ref{H2} and $\Lambda(\cdot)$ satisfies  one of \ref{H6}-\ref{H7}.  Assume that the initial condition
	 term $I_0(t,x)$ is bounded, namely, 
	 there is a positive constant $C$ such that $\sup_{(t,x)\in 
	 \RR_+\times \in \RR^d}
	 |I_0(t,x)|\le C$.      Then there is a unique   mild solution $u(t,x)$ is satisfying  \eqref{eq.MildSolL}. Moreover,   there are some constants $C_1$ and $C_2$ do not depend on $t$, $p$ and $x$ such that \eqref{eq.UMB_thm} holds.
\end{thm}

\begin{proof}[Proof of Theorem \ref{thm.UMB} and Theorem \ref{thm.UMB2}]
As indicated in \cite{Hu19}, there are mainly three   
approaches to obtain the upper moments, effective in different situations.  In our case, we choose 
to use the approach of combining chaos expansion and hypercontractivity inequality. 
	
\textbf{Step 1:} We shall show the upper bound under assumptions \ref{A3}, \ref{H2} and one of \ref{H3}-\ref{H5}. In the following, we shall only prove the case \ref{H3}. The cases \ref{H4} and \ref{H5} can be done similarly. Recall the Wiener-It\^o  chaos expansion \eqref{eq.WCE} for the mild solution to \eqref{eq.MildSolL'}  
\begin{equation*}
	u(t,x)=I_0(t,x) +\sum_{n=1}^{\infty} I_n(f_n(\cdot,t,x))\,,
\end{equation*}
where $f_n(\cdot, t,x)$ is given by
\eqref{eq.f_n_symm}. Denote $u_n(t, x)=I_n(f_n(\cdot,t,x))$.  
 Then it follows from the It\^o isometry for the multiple Wiener-It\^o  
  integral (e.g. \cite{H2017})   that 
\begin{align*}
	\|u_n(t,x) \|_{L^2}^2  
 	 &=\EE   |I_n(f_n(\cdot;t,x)) |^2 \\
&= n! \|{f}_n (\cdot;t,x)\|^2_{\cH^{\otimes n}} \,. 
\end{align*}
To compute the above norm let us 
denote  $\vec{t}=(t_1,\cdots,t_n)$, $\vec{s}=(s_1,\cdots,s_n)$,  $\vec{x}=(x_1,\cdots,x_n)$, $\vec{y}=(y_1,\cdots,y_n)$  and
\begin{equation*}
	\Psi_{n}(\vec{t},\vec{s}):=\int_{\RR^{2nd}}  {f}_n (\vec{t},\vec{x};t,x) \prod_{j=1}^n \Lambda(x_j-y_j)  {f}_n (\vec{s},\vec{y};t,x)d\vec{x}d\vec{y} \,.
\end{equation*}
Then,   we have
\begin{eqnarray}
&& \|u_n(t,x) \|_{L^2}^2  =
n!\|{f}_n (\cdot;t,x)\|^2_{\cH^{\otimes n}}\nonumber\\
&&\qquad \qquad = \frac{1}{n!}\Phi_n(t):=\frac{c_H^n}{n!} \int_{[0,t]^{2n}} \prod_{j=1}^n \gamma(t_j-s_j) \Psi_{n}(\vec{t},\vec{s}) d\vec{t}d\vec{s}\,. \label{Def.Phi_up}
\end{eqnarray} 
By the Cauchy-Schwarz inequality $\Psi_{n}(\vec{t},\vec{s})\leq  \lk\Psi_{n}(\vec{t},\vec{t}) \Psi_{n}(\vec{s},\vec{s})\rk^{1/2}$ and Hardy-Littlewood-Sobolev inequality \cite[Inequality (2.4)]{HN2009}, we obtain  with $\gamma=2-2H$ (or $H=1-\frac{\ga}{2}$)   from \eqref{Def.Phi_up}
\begin{align*}
	\Phi_n(t) \ls & \int_{[0,t]^{2n}} \prod_{j=1}^n \gamma(t_j-s_j) \lk\Psi_{n}(\vec{t},\vec{t}) \Psi_{n}(\vec{s},\vec{s})\rk^{1/2} d\vec{t}d\vec{s} \\
	\ls &\lc \int_{[0,t]^{n}} |\Psi_{n}(\vec{s},\vec{s})|^{1/H}d\vec{s} \rc^{2H} \,.
\end{align*} 
Now we need to resort to  the key assumption  \ref{A3}, i.e. \emph{Hardy-Littelewood-Sobolev type  mass property} \emph{M$(\hbar)$}   to obtain the bound for $\Psi_n$.  
Repeatedly using \ref{A3} (namely, \eqref{A.F_upper}), we have  
\begin{align}\label{Est.Psi_n}
\begin{split}
	\Psi_{n}(\vec{s},\vec{s}) \leq& \int_{\RR^{2nd}}  {f}_n (\vec{s},\vec{x};t,x) \prod_{j=1}^n \Lambda(x_j-y_j)  {f}_n (\vec{s},\vec{y};t,x) d\vec{x}d\vec{y} \\
	 \ls\,& \prod_{j=1}^n |s_{\sigma(j+1)}-s_{\sigma(j)}|^{\hbar} \1_{\{0<s_{\sigma(1)}<\cdots<s_{\sigma(n)}<t\}}\,,
\end{split}
\end{align}
where $\hbar>-1$ and $0<s_{\sigma(1)}<\cdots<s_{\sigma(n)}<s_{\sigma(n+1)}=t$.  Denote the simplex $\prod_n(t)
=\left\{(s_1, \cdots, s_n)\,; 0<s_1<\cdots<s_n<t\right\}$.  Then, 
using the bound we just obtained for $\Psi_n$, we obtain the upper bound for $\Phi_n(t)$:
\begin{align*}
	\Phi_n(t)\leq& C_H^n \lc n!\int_{\prod_n(t)} \prod_{j=1}^n |s_{j+1}-s_{j}|^{\hbar /2H} d\vec{s} \rc^{2H} \\
	\leq& C_H^n \lc n! \frac{t^{n\hbar /2H+n}}{\Gamma(n\hbar 
	/2H+n+1)}\rc^{2H} \simeq   C_H^n \frac{t^{n(\hbar +2H)}}{(n!)^{\hbar}}\,
\end{align*}
by Stirling's formula for Gamma function.

As a result, the second moment can be estimated as
 \begin{align*}
	\|u_n (t,x)\|_{L^2}^2 =\frac{1}{n!} \Phi_n(t)
	\leq&  C_H^n \frac{t^{n(\hbar+2H)}}{(n!)^{\hbar+1}} 
	 \,.
\end{align*}
It is now easy to bound  the $p$-th moment from the above  second moment bound  by using the hypercontractivity inequality (e.g. \cite[p.54, Theorem 3.20]{H2017})
\begin{align*}
	\|u_n (t,x)\|_{L^p}
	\leq&  (p-1)^{n/2} \|u_n( t,x))\|_{L^2} \\
	\leq&   C_H^n (p-1)^{n/2} \lk\frac{t^{n(\hbar+2H)}}{(n!)^{\hbar+1}}\rk^{1/2} \\
\end{align*}
Thus
\begin{align*}
	\|u  (t,x)\|_{ p}
 \leq& C+\sum_{n=1}^{\infty} \|u_n(t,x))\|_{ p} \\ 
	\leq& C+\sum_{n=1}^\infty   C_H^n (p-1)^{n/2} \lk\frac{t^{n(\hbar +2H)}}{(n!)^{\hbar+1}}\rk^{1/2} \\
 	\leq& C \exp\lc C\cdot t^{\frac{\hbar+2H}{\hbar+1}} 
 	(p-1)^{\frac{1}{\hbar+1}}\rc\,. 
\end{align*}
This means  $\EE[|u(t,x)|^p]\leq C_1 \exp\lc C_2\cdot
 t^{1+\frac{1-\gamma}{\hbar+1}} p^{1+\frac{1}{\hbar+1}}\rc$ 
 for some positive constants $C_1$ and $C_2$ and hence we conclude the proof of  
  Theorem \ref{thm.UMB}.  
  
\textbf{Step 2:} We shall show the upper bound under assumptions \ref{A3'}, \ref{H2} and one of \ref{H6}-\ref{H7}. Denote
\begin{align}\label{eq.f_n_symm_Q}
& f_n^Q(\cdot,t,x):=  {f_n}(t_1,x_1,\cdots,t_n,x_n,t,x)\\
=&\frac{1}{n!} \sum_{\sigma\in S_n} Q_{t-t_{\sigma(n)}}(x-x_{\sigma(n)})Q_{t_{{\sigma(n)}}-t_{{\sigma(n-1)}}}(x_{{\sigma(n)}}-x_{{\sigma(n-1)}})	\cdots \nonumber\\
&\qquad\quad \times Q_{t_{{\sigma(2)}}-t_{{\sigma(1)}}}(x_{{\sigma(2)}}-x_{{\sigma(1)}})I  _0(t_{{\sigma(1)}},x_{{\sigma(1)}}) \,.  
\end{align}
Namely, we replace $G$ in the expression of $f_n(\cdot, t,x)$ by $Q$.  Then by the positivity of $G$,   $\La$,
and the fact that $G\le Q$,   we have
\begin{align*}
	\Psi_{n}(\vec{s},\vec{s}) \ls  & \int_{\RR^{2nd}} f_n (\vec{s},\vec{x};t,x) \prod_{j=1}^n \Lambda(x_j-y_j) f_n (\vec{s},\vec{y};t,x) d\vec{x}d\vec{y} \\
	\ls  & \int_{\RR^{2nd}} f_n^Q (\vec{s},\vec{x};t,x) \prod_{j=1}^n \Lambda(x_j-y_j) f_n^Q (\vec{s},\vec{y};t,x) d\vec{x}d\vec{y} \\
	\ls& \int_{\RR^{nd}} \lt| \cF[f_n (\vec{s},\cdot;t,x)](\xi)\rt|^2 |\mu|(d\vec{\xi})\, \ls\, \prod_{j=1}^n |s_{\sigma(j+1)}-s_{\sigma(j)}|^{\hbar}   \,.
\end{align*}
Thus, we get $\EE[|u(t,x)|^p]\leq C_1 \exp\lc C_2\cdot t^{1+\frac{1-\gamma}{\hbar+1}} p^{1+\frac{1}{\hbar+1}}\rc$ for some positive constants $C_1$ and $C_2$.
\end{proof}

\section{Feynman diagram formula}\label{s.5}
Now we turn to the proof of Theorem \ref{thm.LMB},  i.e., the lower bounds for the moments.  
The main   difficulty    is the lack of the Feynman-Kac formula   for general partial differential operators. 
To get around of this difficulty our strategy is a brutal force one. 
We try to handle the $p$-th moment of $u(t,x)$ directly,  where $p$ is an
arbitrary  positive integer and $u$ is the mild solution to 
\eqref{eq.LDiffu}, given by its chaos  expansion \eqref{eq.WCE}.  
Since the solution is an infinite  sum of multiple Wiener-It\^{o} integrals
so we need first to use   the product formulas of the  
multiple Wiener-It\^{o} integrals (with respect to Gaussian  noise).
This is called the Feynman diagram formulas and they  can be found 
in Theorem 5.7 and Theorem 5.8 in \cite{H2017}      
  (for general Gaussian noise case), Theorem 10.2 in \cite{M2013}
 or Theorem 5.3 in \cite{M2014} 
 (for White noise cases).  In this section we shall present this formula ``graphically"  so that we can keep track the  terms. 

Recall that the Gaussian space  $\cH$ in our situation is the   Hilbert space 
obtained by  the completion of $\cD(\R+\times\RR^{d})$ with 
respect to the scalar product defined by \eqref{Prod_H}.
Since the work of    \cite{M2013}
 or   \cite{M2014}  are for the ``White noise" case, we will follow
  the product formula of  \cite[Theorem 5.7]{H2017}.  Since we are only
interested in the expectation of the product of   multiple integrals
and since $\EE \left[I_k(f)\right] =0$ for all  $k\ge 1$ we only need to sum the terms
with  $|\ga|=0$ in  
\cite[Theorem 5.7]{H2017} when we take 
the expectation of the left hand side of \cite[Equation 5.3.5]{H2017} (The notation $\ga$ used
in \cite{H2017} is different than the one used in this paper).

To visualize these summation terms   graphically, we recall the 
concept of diagram associated with only these terms.   
A  Feynman diagram $D$ is a set  of some vertices and some edges 
connecting them so that the vertices are arranged into some 
finite  rows and each row contains some finite many vertices. 
The set of vertices of the diagram
$D$ can then be represented by  $\sV(\cD)=\{(k,r):1\le k\le m, 1\leq r \leq n_m\}$. 
We use $\sE(\cD)=\{[(\overline{k},\overline{r}),(\underline{k},\underline{r})]: 
\overline{k}<\underline{k}\}$ denote the set of all edges of a diagram $\cD$, 
where $\overline{k}<\underline{k}$ means 
$(\overline{k},\overline{r})$ is the upper  (row) 
and $(\underline{k},\underline{r})$ is the lower (row) end point of a edge. 
The strict inequality is important here since two vertices 
in the same row are not allowed to form an edge. 
For an edge $[(\overline{k},\overline{r}),(\underline{k},\underline{r})]\in 
\sE(\cD)$, we call $(\overline{k},\overline{r})$ the upper vertice and 
$(\underline{k},\underline{r})$  the lower vertice of the edge 
 and   we call a vertice associates with an  edge if it is either upper
 or  lower vertice of the edge.  
We use $\overline{\sV}(\cD)$ and $\underline{\sV}(\cD)$ 
to the sets of all upper and lower vertices, respectively.  We require that one vertice
associates with at most one edge. Thus
we have $\overline{\sV}(\cD)\cap \underline{\sV}(\cD)=\emptyset$.  
After taking the expectation of 
\cite[Equation 5.3.5]{H2017}, the terms will be significantly reduced. To account the remaining terms  we only need to consider the 
 following special diagrams. 
\begin{defn}  A diagram $\cD=(\sV(\cD), \sE(\cD))$ is called \emph{admissible} if  every 
vertice is associated with one and only one edge. The set of all \emph{admissible diagrams} associated with the vertices $\left\{(k, r), 1\le k\le m, 1\le r\le n_k\right\}$ 
is denoted by $\DD(n_1, \cdots, n_m)$.  
\end{defn}
It is clear that if a diagram  $\cD$ is admissible then 
$n_1+\cdots+n_m=2|\sE(\cD)|$, in particular, 
$n_1+\cdots+n_m$ is an even integer.  

 
Let $f_k:(\RR_+\times \RR^d) ^{n_k}\rightarrow \RR$, $k=1, \cdots, m$ 
be some given measurable functions. 
Associated with these functions we   
have naturally  the set of Feynman diagrams   
 $\tilde \DD(f_1, \cdots, f_m)$. The correspondence  is described  as follows.
 Each Feynman diagram 
 $\cD\in \tilde \DD(f_1, \cdots, f_m) $ contains 
 $m$ rows, corresponding to $f_1, \cdots, f_m$, and 
 the $k$-th row
 of $\cD$ contains $n_k$ vertices,  which is the number of independent variables of the function $f_{k}$.     
We use  $\DD(f_1, \cdots, f_m)$ 
to denote the set of all admissible Feynman diagrams
associated with $f_1, \cdots, f_m$. 

For the sake of convenience we consider $(t,x)$ as 
one vector independent variable, where $t\ge 0, x\in \RR^d$.  
So  we shall say that $f_k:(\RR_+\times \RR^d)^{n_k}\rightarrow \RR$ has $n_k$
independent (vector) variables.  
From the functions $f_k:(\RR_+\times \RR^d)^{n_k}\rightarrow \RR$, $k=1, \cdots, m$,
we define their concatenation 
$f_1\circ \cdots \circ f_m$ as a function of $n_1+\cdots+n_m$ 
independent vector variables. We name the $n_k$ independent variables of the 
function $f_k$ by $(t_{(k,1)},x_{(k, 1)}), \cdots, (t_{(k,_{n_k})},x_{(k, n_k)})$, 
associated with the $k$-th row vertices.   Thus for an  
admissible Feynman diagram  $\cD\in   \DD(f_1, \cdots, f_m) $, 
the concatenation $f_1\circ \cdots \circ f_m$ is a function of $
n_1+\cdots+n_m$ vector variables and we write  it as 
\[
f_1\circ \cdots \circ f_m((t_{\overline{\sV}(\cD)}, x_{\overline{\sV}(\cD)}),
(t_{\underline{\sV}(\cD)}, x_{\underline{\sV}(\cD)}))\,.
\] 
The edges in $\cD\in   \DD(f_1, \cdots, f_m) $
are  used to  form the (tensor) scalar product in the Gaussian space $\cH$ of the above concatenation.  Here is the detail of this construction. 
If $[(\overline{k},\overline{r}),(\underline{k},\underline{r})]$ 
is an edge of the diagram $\cD$,  then we form a  factor
\[
 \gamma(t_{(\overline{k},\overline{r})}-t_{(\underline{k},\underline{r})})
 \Lambda(x_{(\overline{k},\overline{r})}-  x_{(\underline{k},\underline{r})})
\,.
\]
For the set of $\sE(\cD)$, we denote the product of all above factors   as
\begin{eqnarray}
&&\gamma(t_{\overline{\sV}(\cD)}-t_{\underline{\sV}(\cD)} )
 \Lambda(x_{\overline{\sV}(\cD)}-x_{\underline{\sV}(\cD)}) \nonumber\\
&&\qquad   =\prod_{[(\overline{k},\overline{r}),(\underline{k},\underline{r})]\in 
\sE(\cD)}\gamma(t_{(\overline{k},\overline{r})}-t_{(\underline{k},\underline{r})})
 \Lambda(x_{(\overline{k},\overline{r})}-  x_{(\underline{k},\underline{r})})
\,.\label{e.5.1} 
\end{eqnarray}

With these notations,   we define  finally a  real number
associated with  $f_1, \cdots, f_m$ and associated with an admissible diagram $\cD$ as follows: 
\begin{align}\label{eq.FDiag}
	F_{\cD}(f_1,&\cdots,f_m) =\int   f_1\circ \cdots \circ 
	f_m((t_{\overline{\sV}(\cD)}, x_{\overline{\sV}(\cD)}),
(t_{\underline{\sV}(\cD)}, x_{\underline{\sV}(\cD)})) \nonumber\\
&
\gamma(t_{\overline{\sV}(\cD)}-t_{\underline{\sV}(\cD)} )
 \Lambda(x_{\overline{\sV}(\cD)}-x_{\underline{\sV}(\cD)})
 dt_{\overline{\sV}(\cD)}
 dt_{\underline{\sV}(\cD)}  dx_{\overline{\sV}(\cD)}
 dx_{\underline{\sV}(\cD)}\,. 
\end{align}
To illustrate  the above notation, we give one example. 
\begin{exmp}\label{Exmp.white}
Given three functions of four independent of variables. 
Let us take the following admissible  diagram $\cD\in {\bD}(4,4,4)$ 
in Figure \ref{Fig.1} as an example. 
The upper vertices are colored in red and the lower vertices are colored in blue.
 The upper and lower  variables are gives as follows: 
\begin{align*}
	x_{\overline{\sV}(\cD)}=\,&\{(1,1), (1,2), (1,3), (1,4), (2,1), (2,4)\}\,, \\
	x_{\underline{\sV}(\cD)}=\,&\{(2,2), (2,3), (3,1), (3,2), (3,3), (3,4)\}\,,
\end{align*}
The corresponding set of edges of this diagram     is
\begin{align*}
\sE(\cD) =&\{[(1,1), (3,1)], [(1,2), (2,3)],  [(1,3),((3,4)],\\ 
&\qquad  [(1,4), (2,2)], [(2,1), (3,2)], [(2,4), (3,3)]
 \}\,. 
\end{align*} 
In this case, $n_1=n_2=n_3=4$.  It is easy to see that 
$|\sE(\cD)|=6=(n_1+n_2+n_3)/2$ and 
\begin{align*}
	F_{\cD}(f_1, f_2,&f_3) \\ 
	=&\int f_1(
	(t_{(1,1)}, x_{(1,1)}),(t_{(1,2)}, x_{(1,2)}),(t_{(1,3)}, x_{(1,3)}),(t_{(1,4)}, x_{(1,4)}))
	\\
	&\qquad \cdot 
	f_2(
	(t_{(2,1)}, x_{(2,1)}),(t_{(2,2)}, x_{(2,2)}),(t_{(2,3)}, x_{(2,3)}),(t_{(2,4)}, x_{(2,4)}))
	\\
&\qquad \cdot f_3(
	(t_{(3,1)}, x_{(3,1)}),(t_{(3,2)}, x_{(3,2)}),(t_{(3,3)}, x_{(3,3)}),(t_{(3,4)}, x_{(3,4)}))
	\\ 
	&\qquad \cdot \ga (t_{(1,1)}-t_{(3,1)})\ga (t_{(1,2)}-t_{(2,3)})\ga (t_{(1,3)}-t_{(3,4)})\\
	&\qquad \cdot \ga (t_{(1,4)}-t_{(2,2)})\ga (t_{(2,1)}-t_{(3,2)})\ga (t_{(2,4)}-t_{(3,3)})
	\\ 
	&\qquad \cdot \La(x_{(1,1)}-x_{(3,1)})\La (x_{(1,2)}-x_{(2,3)})\La(x_{(1,3)}-x_{(3,4)})\\
	&\qquad \cdot \La (x_{(1,4)}-x_{(2,2)})\La (x_{(2,1)}-x_{(3,2)})\La (x_{(2,4)}-x_{(3,3)}) dtdx\,, 
\end{align*}
where  $dt=dt_{(1,1)} \cdots dt_{(3,4)}$ and similar notation for $dx$.   
Notice that the Feynman diagrams
are used to track the terms and to provide guidance   for
the variables inside  $\ga$ and $\La$.  

Let us also notice that the operation $F_\cD$ can also be defined 
for any $f_k\in \cH^{\otimes n_k}$, $k=1, \cdots, m$, 
which may contain  measures or generalized functions.  

\begin{figure}[htb]
\centering
\begin{tikzpicture}[scale=2.8]
\tikzstyle{vertex}=[circle, black, inner sep=0pt, minimum size=5pt]
      \foreach \i in {1,...,4}
      {
         \path (\i,0) coordinate (X\i);
         \node[vertex,fill=blue] (X\i) at (\i,0) [label=left:$t_{(3,\i)}$] [label=right:$x_{(3,\i)}$] {};
      }
      \foreach \j in {1,...,4}
      {
         \path (\j,1) coordinate (Y\j);
      }
      \node[vertex,fill=red] (Y1) at (1,1) [label=left:$t_{(2,1)}$] [label=right:$x_{(2,1)}$] {};
      \node[vertex,fill=blue] (Y2) at (2,1) [label=left:$t_{(2,2)}$] [label=right:$x_{(2,2)}$] {};
      \node[vertex,fill=blue] (Y3) at (3,1) [label=left:$t_{(2,3)}$] [label=right:$x_{(2,3)}$] {};
      \node[vertex,fill=red] (Y4) at (4,1) [label=left:$t_{(2,4)}$] [label=right:$x_{(2,4)}$] {};
      \foreach \k in {1,...,4}
      {
         \path (\k,2) coordinate (Z\k);
		 \node[vertex,fill=red] (Z\k) at (\k,2) [label=left:$t_{(1,\k)}$] [label=right:$x_{(1,\k)}$] {};
      }
      \path
       (Z1) edge[bend right=20] (X1) (Z2) edge (Y3) (Z3) edge (X4) (Z4) edge (Y2)
       (Y1) edge (X2) (Y4) edge (X3);
\end{tikzpicture}
\caption{An example the admissible diagram}
\label{Fig.1}
\end{figure}
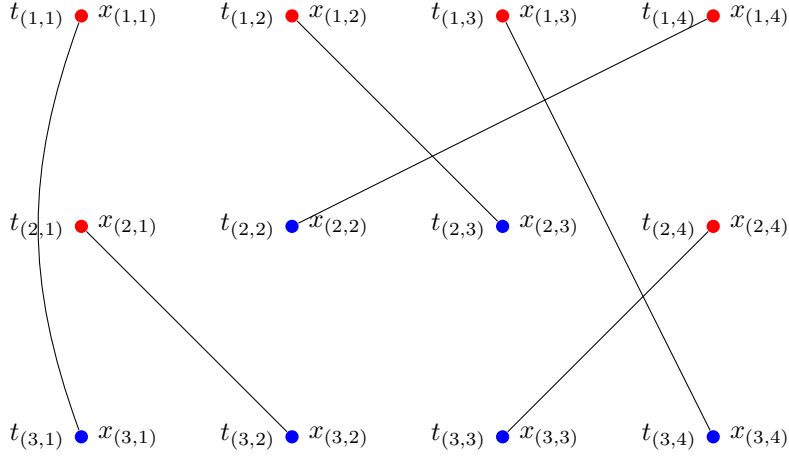
\end{exmp}

\begin{thm}\label{Thm.EofMWI_color}
	For $f_1\in \sH^{\otimes n_1}$,\dots, $f_m\in \sH^{\otimes n_m}$, we have 
	\begin{equation}\label{eq.EE_Diag}
		\EE[I _{n_1}(f_1)\cdots I _{n_m}(f_m)]=\sum_{\cD\in  {\bD}(f_1,\cdots,f_m)} F  _{\cD}(f_1,\cdots,f_m)\,.
	\end{equation}
where $F _{\cD}(f_1,\cdots,f_m)$ is given by
\eqref{eq.FDiag}. 
\end{thm}
\begin{proof}
This theorem is a consequence of 
\cite[Theorem 5.7]{H2017} when we take the expectation of \cite[Equation 5.3.5]{H2017} and notice that now the scalar product  of  $\cH$ is defined by  \eqref{Prod_H}.     
\end{proof}

%
%

If we apply the above formula to the $f_n$ defined by \eqref{eq.f_n_symm}, we have 
\begin{thm}\label{cor.EE_Diag}
Let  $ {f_n}(\cdot,t,x)$ be defined by \eqref{eq.f_n_symm}
and let  $I _n( {f_n}(\cdot,t,x))$ be the associated multiple Wiener-It\^o integral. 
Then 
	\begin{equation} 
		\begin{split}
			\EE \Blk I _{n_1}( &{f_{n_1}}(\cdot,t,x))\cdots I _{n_m}( {f_{n_m}}(\cdot,t,x))\Brk \\
		=&\sum_{\cD\in  {\bD({f_{n_1}} ,\cdots, {f_{n_m}})}} 
		F _{\cD}\lt(  {f_{n_1}} ,\cdots, {f_{n_m}}  \rt)  \\
		=&\sum_{\cD\in   \bD({f_{n_1}} ,\cdots, {f_{n_m}}) }\int \prod_{j=1}^{m}\prod_{r=1}^{n_j} 
		G_{t_{(j,r+1)}-t_{(j,r)} }\lt(x_{(j,r+1)}-x_{(j,r)}\rt) \1_{\{0<t_{ (j, 1)}<
		\cdots< t_{ (j, n_j)}\}}  \\
		&\qquad\qquad\qquad\qquad \times\gamma\Blc t_{\overline{\sV}(\cD)}-t_{\underline{\sV}(\cD)}\Brc \Lambda\Blc x_{\overline{\sV}(\cD)}-x_{\underline{\sV}(\cD)}\Brc dt_{\cD}dx_{\cD}\,, \label{eq.EE_Diag.G}
		\end{split}
	\end{equation}
where we use  the notations $t_{(j,n_j+1)}=t$ and $x_{(j,n_j+1)}=x$ for all $1\leq j\leq m$
and  $\gamma(\cdot)$ and $\Lambda(\cdot)$ are defined as \eqref{e.5.1}.
\end{thm}
\begin{proof}
%
We only need to prove the second equality in \eqref{eq.EE_Diag.G}.  We may only consider the time variable without loss of generality (i.e. $d=0$). Namely, we reduce  the symmetric function $ f_{n} (t_1,x_1,\cdots,
t_n,x_n;t,x)$ to  the symmetric function   $ f_{n} (t_1,\cdots,t_n;t)$. 
Then what we need to show is the following equality for any $n_1,\dots,n_m$ and for any corresponding 
	admissible  diagrams $ {\bD}$, \eqref{eq.EE_Diag.G}   holds true.  
	We shall prove  \eqref{eq.EE_Diag.G}   recursively on $n$.
	Denote the function of \eqref{eq.f_n} by  $f_{n} (t_1,\cdots,t_n;t)$ and its symmetrization by $\tilde f_{n} (t_1,\cdots,t_n;t)$.  Then 
	\begin{align}\label{eq.EE_Diag.G1}
		&\sum_{\cD\in  \bD({\tilde f_{n_1}} ,\cdots, {\tilde f_{n_m}}) } F_{\cD}\lt(    \tilde f_{n_1} (\cdot,t),\cdots, {\tilde  f_{n_m}}(\cdot,t) \rt)  \nonumber\\
		=\,&\sum_{\cD\in  \bD({f_{n_1}} ,\cdots, {f_{n_m}})}\int \prod_{j=1}^{m} {\tilde f_{n_j}}(t_{(j,1)},\cdots,t_{(j,n_j)};t) \times\gamma\Blc t_{\overline{\sV}(\cD)}-t_{\underline{\sV}(\cD)}\Brc  dt_{\cD} \nonumber\\
		=\,&\sum_{\cD\in  \bD({f_{n_1}} ,\cdots, {f_{n_m}})}  \frac{1}{n_1!}\sum_{\sigma\in S_{n_1}} 
		 \int  {f}_{n_1} (t_{ (1,\sigma(1))},\cdots,t_{(1,\sigma(n_1))};t)\nonumber \\
		&\qquad\qquad\qquad\quad \cdot \prod_{j=2}^{m}  {\tilde f_{n_j}}(t_{(j,1)},\cdots,t_{(j,n_j)};t) \times\gamma\Blc t_{\overline{\sV}(\cD)}-t_{\underline{\sV}(\cD)}\Brc  dt_{\cD} \nonumber\\
		=\,& \sum_{\cD\in  \bD({f_{n_1}} ,\cdots, {f_{n_m}})}  \frac{1}{n_1!}\sum_{\sigma\in S_{n_1}} 
		{I_{\si(1), \cdots, \si(n_1), \cD}}\,, 
	\end{align} 
	where $S_{n_1}$ denotes the set of all permutations of $\{1, 2, \cdots, n_1\}$
	and $I_{\si(1), \cdots, \si(n_1), \cD}$ denotes the above integral. 
	Suppose that  $\cD\in  \bD({f_{n_1}} ,\cdots, {f_{n_m}})$ is a Feynman diagram.
Then there there are $(j_1, r_1), \cdots, (j_{n_1}, r_{n_1})$ such that the following edges
\[
[(1,1), (j_1, r_1)], \cdots, [(1, n_1), (j_{n_1}, r_{n_1})]  
\]
are in $\sE$. In this diagram $\cD$, we replace the above edges by the following ones 
\[
[(1,\si(1)), (j_1, r_1)], \cdots, [(1, \si(n_1)), (j_{n_1}, r_{n_1})]  
\]
and retain all other edges unchanged. Then we obtain another diagram  $\cD_\si$.  
See   Figure \ref{Fig.1_1} for an illustrating  example.
\begin{figure}[htb]
\begin{minipage}{.5\textwidth}
\centering
\begin{tikzpicture}[scale=1.2]
\tikzstyle{vertex}=[circle, draw, inner sep=0pt, minimum size=4pt]
      \foreach \i in {1,...,4}
      {
         \path (\i,0) coordinate (X\i);
         \vertex[fill] (X\i) at (\i,0) [label=right:$t^{(3)}_{\i}$] {};
      }
      \foreach \j in {1,...,4}
      {
         \path (\j,1) coordinate (Y\j);
		 \vertex[fill] (Y\j) at (\j,1) [label=right:$t^{(2)}_{\j}$] {};
      }
      \foreach \k in {1,...,4}
      {
         \path (\k,2) coordinate (Z\k);
		 \vertex[fill] (Z\k) at (\k,2) [label=right:$t^{1}_{\k}$] {};
      }
      \path
       (Z1) edge[bend right=20] (X1) (Z2) edge (Y3) (Z3) edge[bend left=15] (X4) (Z4) edge[bend right=10] (Y2)
       (Y1) edge (X2) (Y4) edge (X3);
\end{tikzpicture}
\end{minipage}%
\begin{minipage}{.5\textwidth}
\centering
\begin{tikzpicture}[scale=1.2]
\tikzstyle{vertex}=[circle, draw, inner sep=0pt, minimum size=4pt]
      \foreach \i in {1,...,4}
      {
         \path (\i,0) coordinate (X\i);
         \vertex[fill] (X\i) at (\i,0) [label=right:$t^{(3)}_{\i}$] {};
      }
      \foreach \j in {1,...,4}
      {
         \path (\j,1) coordinate (Y\j);
		 \vertex[fill] (Y\j) at (\j,1) [label=right:$t^{(2)}_{\j}$] {};
      }
      \foreach \k in {1,...,4}
      {
         \path (\k,2) coordinate (Z\k);
		 \vertex[fill] (Z\k) at (\k,2) [label=right:$t^{1}_{\sigma(\k)}$] {};
      }
      \path
       (Z2) edge[bend left=15] (X1) (Z1) edge (Y3) (Z4) edge[bend left=15] (X4) (Z3) edge (Y2)
       (Y1) edge (X2) (Y4) edge (X3);
\end{tikzpicture}
\end{minipage}%
\caption{The left is $\sigma=\{1,2,3,4\}$ and the right is $\sigma=\{2,1,4,3\}$}
\label{Fig.1_1}
\end{figure}
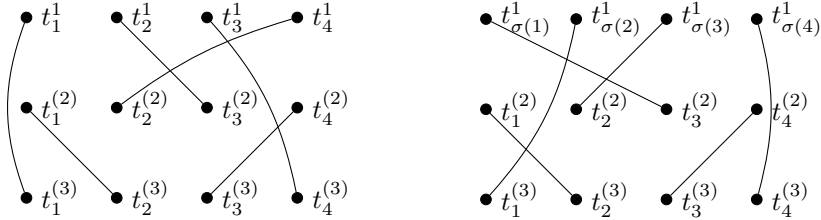
This transformation $\cD\to \cD_\si$ has the following properties.
\begin{enumerate}
\item[(i)] If $\cD$ is an admissible diagram, so is $\cD_\si$.
\item[(ii)] For any fixed  permutation $\si$, the mapping  $\cD\to \cD_\si$ is a bijection  from
$\bD(f_{n_1}, \cdots, f_{n_m})$ to itself.
\item[(iii)]  $\gamma\Blc t_{\overline{\sV}(\cD_{\sigma})}-t_{\underline{\sV}(\cD_{\sigma})}\Brc$ 
remains  unchanged: 
	\[
	 \gamma\Blc t_{\overline{\sV}(\cD_{\sigma})}-t_{\underline{\sV}(\cD_{\sigma})}\Brc=
	 \gamma\Blc t_{\overline{\sV}(\cD )}-t_{\underline{\sV}(\cD )}\Brc\,. 
	\]
\end{enumerate}
These properties imply 
\[
 \sum_{\cD\in  \bD({f_{n_1}} ,\cdots, {f_{n_m}})}    
		I_{\si(1), \cdots, \si(n_1), \cD}
		=
		 \sum_{\cD\in  \bD({f_{n_1}} ,\cdots, {f_{n_m}})}    
		I_{1, \cdots, n_1, \cD}
\,.
\]
Substituting it to   \eqref{eq.EE_Diag.G1}  we have
\begin{align}\label{eq.EE_Diag.G1}
		&\sum_{\cD\in  \bD({\tilde f_{n_1}} ,\cdots, {\tilde f_{n_m}}) } F_{\cD}\lt(  {\tilde f_{n_1}}(\cdot,t),\cdots, {\tilde f_{n_m}}(\cdot,t) \rt)  \nonumber\\ 
		=\,&\sum_{\cD\in  \bD({f_{n_1}} ,\cdots, {f_{n_m}})}\int   f_{n_1 }(t_{(1,1)},\cdots,t_{(1,n_1)};t) \\
		&\qquad\qquad\qquad\qquad \cdot \prod_{j=2}^{m} {\tilde f_{n_j}}(t_{(j,1)},\cdots,t_{(j,n_j)};t) \times\gamma\Blc t_{\overline{\sV}(\cD)}-t_{\underline{\sV}(\cD)}\Brc  dt_{\cD}\,. \nonumber
	\end{align} 
%
This can be used to prove  the theorem by induction. 
\end{proof}

\begin{exmp}  
The above formula \eqref{eq.EE_Diag.G} can be used to compute all moments
of a chaos expansion.  This will be done in the next section when we prove the 
lower moment bounds. 
As an example, it is interesting to consider the second moment. 
By orthogonality of multiple 
	Wiener-It\^o  chaos expansion, we have
	\[
	 \EE[|u(t,x)|^2]=1+\sum_{n=1}^{\infty}\EE\lt[|I _{n}( {f_{n}})|^2\rt]\,.
	\]
	Then by \eqref{eq.EE_Diag.G}  in Theorem  \ref{cor.EE_Diag}, one   finds 
	\begin{align}\label{eq.EE_Diag.2nd}
		\EE\Big[|I _{n}&(  {f_{n}})|^2\Big] \nonumber \\
		=&\sum_{\cD\in  {\bD}(n,n)}\int \prod_{j=1}^{2}\prod_{r=1}^{n} G_{t^{(j)}_{r+1}-t^{(j)}_{r}}\lt(x^{(j)}_{r+1}-x^{(j)}_{r}\rt) \1_{\{0<t^{(j)}_{1}<\cdots<t^{(j)}_{n}<t\}} \\
		&\qquad\qquad\qquad \times\gamma\Blc t_{\overline{\sV}(\cD)}-t_{\underline{\sV}(\cD)}\Brc \Lambda\Blc x_{\overline{\sV}(\cD)}-x_{\underline{\sV}(\cD)}\Brc dt_{\cD}dx_{\cD}\,.\nonumber
	\end{align} 
\end{exmp}	
	An example of admissible  diagram $\cD\in  {\bD}(4,4)$ can be illustrated in the Figure \ref{Fig.2_1}. In this diagram, we have $T_{\overline{\sV}(\cD)}:=(t_{\overline{\sV}(\cD)},x_{\overline{\sV}(\cD)})=\{T^{(2)}_{j}:1\leq j\leq 4\}$ colored  in red, $T_{\underline{\sV}(\cD)}:=(t_{\underline{\sV}(\cD)},x_{\underline{\sV}(\cD)})=\{T^{1}_{j}:1\leq j\leq 4\}$ colored  in blue. Moreover, 
\begin{align*}
  \gamma\Blc t_{\overline{\sV}(\cD)}-t_{\underline{\sV}(\cD)}\Brc :=\, &\gamma\Blc t^{(2)}_{1}-t^{1}_{3}\Brc \gamma\Blc t^{(2)}_{2}-t^{1}_{1}\Brc \gamma\Blc t^{(2)}_{3}-t^{1}_{4}\Brc \gamma\Blc t^{(2)}_{4}-t^{1}_{2}\Brc\,, \\
  \Lambda\Blc x_{\overline{\sV}(\cD)}-x_{\underline{\sV}(\cD)}\Brc :=\, &\Lambda\Blc x^{(2)}_{1}-x^{1}_{3}\Brc \Lambda\Blc x^{(2)}_{2}-x^{1}_{1}\Brc \Lambda\Blc x^{(2)}_{3}-x^{1}_{4}\Brc \Lambda\Blc x^{(2)}_{4}-x^{1}_{2}\Brc\,.
\end{align*}
and $\Lambda\Blc x_{\overline{\sV}(\cD)}-x_{\underline{\sV}(\cD)}\Brc$ is also expressed in the same way. Obviously, there are $4!$ such diagram.
\begin{figure}[htb]
\centering
\begin{tikzpicture}[scale=2]
\tikzstyle{vertex}=[circle, black, inner sep=0pt, minimum size=5pt]
      \foreach \i in {1,...,4}
      {
         \path (\i,0) coordinate (X\i);
         \node[vertex,fill=blue] (X\i) at (\i,0) [label=below:$T^{1}_{\i}$] {};
      }
      \foreach \j in {1,...,4}
      {
         \path (\j,1) coordinate (Y\j);
		 \node[vertex,fill=red] (Y\j) at (\j,1) [label=right:$T^{(2)}_{\j}$] {};
      }
      \path
      (X1) edge (Y2) (X2) edge (Y4) (X3) edge[bend left=20] (Y1) (X4) edge[bend left=10] (Y3);     
\end{tikzpicture}
\caption{A admissible diagram $\cD\in {\bD}(4,4)$ with $T^{(j)}_{l}=(t^{(j)}_{l},x^{(j)}_{l})$}
\label{Fig.2_1}
\end{figure}
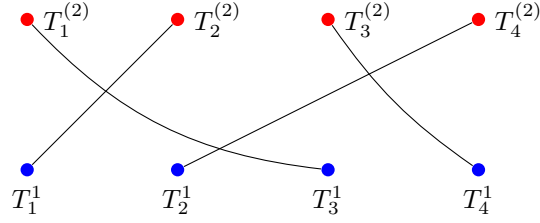

If $\gamma(\cdot)=\delta(\cdot)$, then \eqref{eq.EE_Diag.2nd} reduced to 
\begin{align}\label{eq.EE_Diag.2nd'}
		\EE\Big[|I^{W}_{n}({f_{n}})|^2\Big] 
		=&\int \prod_{r=1}^{n} G_{t_{r+1}-t_{r}}\lt(x_{r+1}-x_{r}\rt)  \Lambda\Blc x_{r}-y_{r}\Brc \\
		&\quad \times G_{t_{r+1}-t_{r}}\lt(y_{r+1}-y_{r}\rt)\cdot \1_{\{0<t_{1}<\cdots<t_{n}<t\}} dt_{\cD}dx dy\,.\nonumber
	\end{align}
This is because the only admissible admissible diagram is the `trivial' one shown in Figure \ref{Fig.2_2} in this case. Otherwise, in some non-trivial admissible diagrams (e.g. the one in Figure \ref{Fig.2_1}) the indicate function $\1_{\{0<t^{1}_{1}<\cdots<t^{1}_{n}<t\}}$ is not compatible with $\1_{\{0<t^{(2)}_{1}<\cdots<t^{(2)}_{n}<t\}}$.
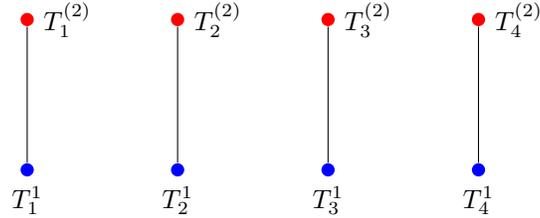
\begin{figure}[htb]
\centering
\begin{tikzpicture}[scale=2]
\tikzstyle{vertex}=[circle, black, inner sep=0pt, minimum size=5pt]
      \foreach \i in {1,...,4}
      {
         \path (\i,0) coordinate (X\i);
         \node[vertex,fill=blue] (X\i) at (\i,0) [label=below:$T^{1}_{\i}$] {};
      }
      \foreach \j in {1,...,4}
      {
         \path (\j,1) coordinate (Y\j);
		 \node[vertex,fill=red] (Y\j) at (\j,1) [label=right:$T^{(2)}_{\j}$] {};
      }
      \path
      (X1) edge (Y1) (X2) edge (Y2) (X3) edge (Y3) (X4) edge (Y4);     
\end{tikzpicture}
\caption{The `trivial' admissible diagram $\cD\in {\bD}(4,4)$}
\label{Fig.2_2}
\end{figure}

\section{Lower moment bounds}\label{s.6}
In this section we use the formula \eqref{eq.EE_Diag.G} to obtain the lower moment bounds
for the mild solution of \eqref{eq.LDiffu}. In the remaining part of the paper we shall use the index 
$(t^{l}_j, x^{l}_j)$ to represent the independent variable
$(t_{(l,j)}, x_{(l,j)})$  associated with the vertice $(l, j)$:
the superscript indicates the row that variable corresponds to
and the subscript indicates the column that variable corresponds to. 
Again, in the following, we shall only prove the case \ref{H3}. The cases \ref{H4} and \ref{H5} can be done similarly.

\begin{proof}[Proof of Theorem \ref{thm.LMB}]
  Let $u(t,x)$ be the mild solution given by
\eqref{eq.WCE}-\eqref{eq.f_n_symm}. 
Let $p$ be an even  positive integer. Applying Theorem  \ref{cor.EE_Diag}, we have
\begin{align}\label{eq.WCE_Prod}
	\EE\Blk \prod_{j=1}^{p}u(t,x_j)\Brk=& \EE\Blk \prod_{j=1}^{p} \sum_{n_j=0}^{\infty} I_{n_j}( {f_{n_j}}(\cdot,t,x_j))\Brk\nonumber \\
	=&\sum_{n_1=0}^{\infty}\cdots \sum_{n_p=0}^{\infty} \EE \Blk I _{n_1}( {f_{n_1}}(\cdot,t,x_j))\cdots I _{n_p}( {f_{n_p}}(\cdot,t,x_j))\Brk\nonumber \\
	=&\sum_{m=0}^{\infty}\sum_{{n_1+\cdots+n_p=2m} \atop{\cD\in  {\bD}(f_{n_1},\cdots,
	f_{n_p})} }F_{\cD}( {f_{n_1}},\cdots, {f_{n_p}})\,.
\end{align}
Notice that the last equality follows from the fact that the number of all vertices 
 of an admissible  diagram $\cD$ must be even.

Our next strategy is to find the suitable lower 
bounds for the term 
\[
\sum_{n_1+\cdots+n_p=2m}\sum_{\cD\in  {\bD}} F_{\cD}
\] 
 in \eqref{eq.WCE_Prod} when  $p$ and $m$ are sufficiently large. 
We shall divide our proof into three steps. 

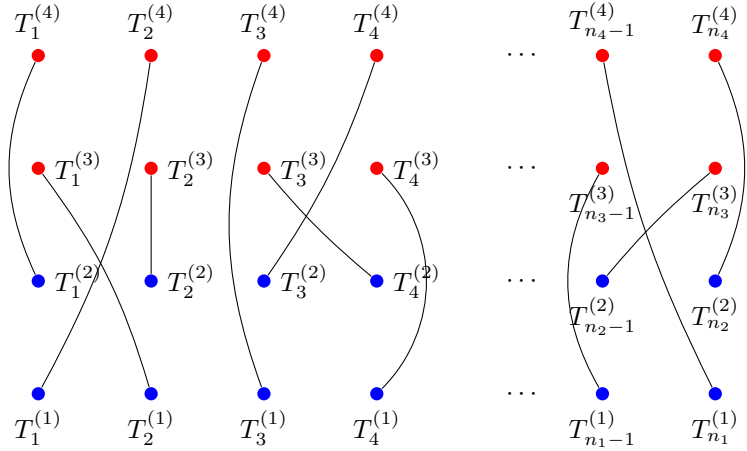
\begin{figure}[htb]
\centering
\begin{tikzpicture}[scale=1.5]
\tikzstyle{vertex}=[circle, black, inner sep=0pt, minimum size=5pt]
      \foreach \i in {1,...,4}
      {
         \path (\i,0) coordinate (X\i);
         \node[vertex,fill=blue] (X\i) at (\i,0) [label=below:$T^{(1)}_{\i}$] {};
      }
      \path (5,0) coordinate (X5);
         \node[vertex] (X5) at (5,0) [label=right:$\cdots$] {};
      \path (6,0) coordinate (X6);
         \node[vertex,fill=blue] (X6) at (6,0) [label=below:$T^{(1)}_{n_1-1}$] {};
      \path (7,0) coordinate (X7);
         \node[vertex,fill=blue] (X7) at (7,0) [label=below:$T^{(1)}_{n_1}$] {};
      \foreach \j in {1,...,4}
      {
         \path (\j,1) coordinate (Y\j);
		 \node[vertex,fill=blue] (Y\j) at (\j,1) [label=right:$T^{(2)}_{\j}$] {};
      }
      \path (5,1) coordinate (Y5);
         \node[vertex] (Y5) at (5,1) [label=right:$\cdots$] {};
      \path (6,1) coordinate (Y6);
         \node[vertex,fill=blue] (Y6) at (6,1) [label=below:$T^{(2)}_{n_2-1}$] {};
      \path (7,1) coordinate (Y7);
         \node[vertex,fill=blue] (Y7) at (7,1) [label=below:$T^{(2)}_{n_2}$] {};
      \foreach \k in {1,...,4}
      {
         \path (\k,2) coordinate (Z\k);
		 \node[vertex,fill=red] (Z\k) at (\k,2) [label=right:$T^{(3)}_{\k}$] {};
      }
      \path (5,2) coordinate (Z5);
         \node[vertex] (Z5) at (5,2) [label=right:$\cdots$] {};
      \path (6,1) coordinate (Z6);
         \node[vertex,fill=red] (Z6) at (6,2) [label=below:$T^{(3)}_{n_3-1}$] {};
      \path (7,1) coordinate (Z7);
         \node[vertex,fill=red] (Z7) at (7,2) [label=below:$T^{(3)}_{n_3}$] {};
      \foreach \l in {1,...,4}
      {
         \path (\l,3) coordinate (W\l);
		 \node[vertex,fill=red] (W\l) at (\l,3) [label=above:$T^{(4)}_{\l}$] {};
      }
      \path (5,3) coordinate (W5);
         \node[vertex] (W5) at (5,3) [label=right:$\cdots$] {};
      \path (6,3) coordinate (W6);
         \node[vertex,fill=red] (W6) at (6,3) [label=above:$T^{(4)}_{n_4-1}$] {};
      \path (7,3) coordinate (W7);
         \node[vertex,fill=red] (W7) at (7,3) [label=above:$T^{(4)}_{n_4}$] {};
      \path
       (W1) edge[bend right=25] (Y1) (W2) edge[bend left=10] (X1) (W3) edge[bend right=20] (X3) (W4) edge[bend left=8] (Y3)
       (Z1) edge[bend left=10] (X2) (Z2) edge (Y2) (Z3) edge[bend right=5] (Y4) (Z4) edge[bend left=45] (X4);
	  \path
	  (W6) edge[bend right=8] (X7) (W7) edge[bend left=25] (Y7)
	  (Z6) edge[bend right=30] (X6) (Z7) edge[bend right=5] (Y6);
\end{tikzpicture}
\caption{A particular scheme when $p=4$}
\label{Fig.3}
\end{figure}

\noindent\textbf{Step 1: }  By the assumption \ref{A1}, namely, all the kernels 
$f_{n_k}$ are nonnegative, to obtain the lower bounds, we can discard any
terms we wish. As  
in \cite{DM2009} we shall keep only those terms such that $n_1=\cdots=n_p$ 
(see the Figure \ref{Fig.3} for a graphical illustration). 
To be more precise, among all the admissible diagrams 
$\cD\in  {\bD}(n_1,\cdots,n_p)$ such that $n_1+\cdots+n_p=2m$, we   take into account only  those diagrams
satisfying the following conditions:
\begin{enumerate}[label=\textbf{(D.\arabic*)}]
    \item\label{Item.Diag_1} We consider only the diagram so that the number of vertices 
    in  each row  are the same. This is, we set
	\begin{equation}\label{eq.cdt_1}
		n_1=\cdots=n_p=\frac{2m}{p}=:m_{p}\,.
	\end{equation}
	\item\label{Item.Diag_2} We set the first $\frac p2$ 
	rows  to be the upper vertices
	$T_{\overline{\sV}(\cD)}:=(t_{\overline{\sV}(\cD)},
	x_{\overline{\sV}(\cD)})$ (which are colored  
	in red in the Figure \ref{Fig.3}), and the remaining 
	rows to be the lower vertices   $T_{\underline{\sV}(\cD)}:=(t_{\underline{\sV}(\cD)},x_{\underline{\sV}(\cD)})$ (which are colored  in blue in the Figure \ref{Fig.3}). 
\end{enumerate}

\begin{rmk}\label{Rmk.perm}
	Fix  the set of  upper   vertices.  Any  permutation of  the lower  vertices corresponds to 
	an admissible diagram in one to one manner. 
	 Then   there are  \emph{$m!$ such admissible diagrams} satisfying the 
	  conditions \ref{Item.Diag_1} and \ref{Item.Diag_2}.	
\end{rmk}

Let $t\in \RR_+$.  Denote  $L=\frac{t}{2(m_p+1)}$, 
$t_j=\frac{j\cdot t}{2(m_p+1)}$ and $I_j=[a_j,b_j]$ 
for $j=1,\dots,m_{p}$, where $a_j=t_j-L/4$ 
and $b_j=t_j+L/4$.  We assure $t^{l}_{j}$ 
is in $I_j$ for $1\leq l\leq p$ and $1\leq j\leq m_p$. 
And we put some restriction on these points such that 
$t^{l}_{j+1}-t^{l}_{j}$ are smaller than $\ep^\sfb $ 
(the same one in \emph{B($\sfa, \sfb$)}) for any $\ep(\leq 1)$, then
	\begin{equation}\label{eq.cdt_2}
		\frac{t}{4 m_{p}}\simeq\frac{t}{4(m_p+1)}
		\leq t^{l}_{j+1}-t^{l}_{j}\leq 
			\frac{t}{m_p+1}\simeq\frac{t}{m_{p}}\leq \ep^\sfb \,.	
	\end{equation}
	Moreover, combining  \eqref{eq.cdt_1} in \ref{Item.Diag_1} 
	and \eqref{eq.cdt_2}, we must have $m_{p}=\frac{2m}{p}\geq \frac{t}{\ep^\sfb }$, which is equivalent to the following conditon:
	\begin{equation}\label{eq.cdt_3}
	 m\geq \frac{p\cdot t}{2\ep^\sfb}\,.
	\end{equation}

\textbf{Step 2: }
Since  $ {f_n}(\cdot,t,x)$ is defined by \eqref{eq.f_n_symm} we can use  \eqref{eq.EE_Diag.G} in Theorem  \ref{cor.EE_Diag}  to bound $F_{\cD}$ in \eqref{eq.WCE_Prod}.

We only consider particular scenario specified  in \textbf{Step 1}.  
We denote the set of all admissible diagrams satisfying satisfying the 
	  conditions \ref{Item.Diag_1} and \ref{Item.Diag_2}  by $ \bD:= \bD(f_{n_1},
	  \cdots, f_{n_m})$. 
When $\cD\in  \bD$, we have 
\begin{align}\label{eq.LMB_0}
	& F_\cD(f_{n_1}, \cdots, f_{n_p})
	=\int \prod_{l=1}^{p}\prod_{j=1}^{m_{p}} 
	G_{t^{l}_{j+1}-t^{l}_{j}}\lt(x^{l}_{j+1},x^{l}_{j}\rt)
	\1_{I_j}(t^{l}_{j}) \1_{\{0<t^{l}_{1}<\cdots<t^{l}_{m_{p}}<t\}}  \\
	&\qquad\qquad\qquad\quad\, \times\gamma\Blc 
	t_{\overline{\sV}(\cD)}-t_{\underline{\sV}(\cD)}
	\Brc \Lambda\Blc x_{\overline{\sV}(\cD)}-x_{\underline{\sV}(\cD)}
	\Brc dt_{\cD}dx_{\cD}\,, \nonumber
\end{align}
with the convention 
that  $x^{l}_{m_p+1}=x$ and $t^{l}_{m_p+1}=t$ for all $1\leq l\leq p$.

It seems very difficult to compute the multiple integral
in \eqref{eq.LMB_0}. We need to find a suitable lower bounds
of the integral that are the main parts and that are 
relatively easier to handle.  Since $\La(x)\to \infty$ when 
$x\to 0$, we shall first bound the above integral 
with respect to the spatial variables $dx_\cD$  
from below by the   integration  
over  small balls $B_{\ep}(x)$ centered at $x=x_1=\cdots=x_p$ with radius $\ep$. 
By the assumption \ref{H3} or \ref{H4},  it is easy to see $\Lambda\Blc x_{\overline{\sV}(\cD)}-x_{\underline{\sV}(\cD)}\Brc \gtrsim \ep^{-m\lambda}$ since $\#\{\overline{\sV}(\cD)\}=\#\{\underline{\sV}(\cD)\}=m$ and since  we always have $|x_i-x_j|\leq 2\ep$ for any $i\in \overline{\sV}(\cD)$ and $j\in \underline{\sV}(\cD)$.  { Similarly, it is obvious to control $\gamma\Blc t_{\overline{\sV}(\cD)}-t_{\underline{\sV}(\cD)}\Brc \gtrsim t^{-m\gamma}$ because $|t_i-t_j|\leq t$ for any $i\in \overline{\sV}(\cD)$ and $j\in \underline{\sV}(\cD)$.  

On each space-time line, there are $m_{p}$ space-time points. 
By \eqref{eq.cdt_2} we have $t^{l}_{j+1}-t^{l}_{j}\leq \ep^\sfb $ for $1\leq l\leq p$ 
and $1\leq j\leq m_p$. }
The \emph{small ball nondegeneracy property} \emph{B($\sfa,\sfb$)} implies
	\[
	{  \int_{B_{\ep}(x)} G_{t^{l}_{j+1}-t^{l}_{j}}\lt(x^{l}_{j+1},x^{l}_{j}\rt) dx^{l}_{j}\geq C\cdot |t^{l}_{j+1}-t^{l}_{j}|^{\sfa}} 
	\]
	if $x^{l}_{j}$ belong to $B_{\ep}(x)$ for all $l$ and $j$. Thus,  on the domain
	\[
	\Omega_\ep:= \cap_{l=1}^{p}\cap_{j=1}^{m_p}\{  x^{l}_{j}\in B_{\ep}(x) \}
	\] 
we have from the simple fact 
$\Lambda\Blc x_{\overline{\sV}(\cD)}
-x_{\underline{\sV}(\cD)}\Brc\gtrsim \ep^{-m\lambda}$:
\begin{align*}
	\int \prod_{l=1}^{p}\prod_{j=1}^{m_{p}} G_{t^{l}_{j+1}-t^{l}_{j}}&\lt(x^{l}_{j+1},x^{l}_{j}\rt)\Lambda\Blc x_{\overline{\sV}(\cD)}-x_{\underline{\sV}(\cD)}\Brc dx_{\cD} \\
	&\gtrsim\,\int_{\Om_\ep} \prod_{l=1}^{p}\prod_{j=1}^{m_{p}} G_{t^{l}_{j+1}-t^{l}_{j}} \lt(x^{l}_{j+1},x^{l}_{j}\rt)\Lambda\Blc x_{\overline{\sV}(\cD)}-x_{\underline{\sV}(\cD)}\Brc dx_{\cD} \\
	&\gtrsim\, \ep^{-m\lambda}\int_{B_{\ep}(x)^{2m}} \prod_{l=1}^{p}\prod_{j=1}^{m_{p}} G_{t^{l}_{j+1}-t^{l}_{j}}\lt(x^{l}_{j+1},x^{l}_{j}\rt) dx_{\cD} \\
	&=\, \ep^{-m\lambda}\int_{B_{\ep}(x)^{2m-1}} \int_{B_{\ep}(x)} G_{t^{1}_{2}-t^{1}_{1}}\lt(x^{1}_{2},x^{1}_{1}\rt) dx^{1}_{1}\\
	&\qquad\qquad\qquad\quad\times \prod_{l=1,j=1 \atop l, j \neq 1}^{p,m_{p}} G_{t^{l}_{j+1}-t^{l}_{j}}\lt(x^{l}_{j+1},x^{l}_{j}\rt) d{x_{\cD}\backslash x^{1}_{1}}  \\
	&\gtrsim\, \ep^{-m\lambda} \lt| t^{1}_{2}-t^{1}_{1} \rt|^{\sfa} \\
	&\qquad\ \times\int_{B_{\ep}(x)^{2m-1}} \prod_{l=1,j=1 \atop l, j \neq 1}^{p,m_{p}} G_{t^{l}_{j+1}-t^{l}_{j}}\lt(x^{l}_{j+1},x^{l}_{j}\rt) d{x_{\cD}\backslash x^{1}_{1}}\,,
\end{align*}
where we used  \eqref{eq.cdt_2} and $d{x_{\cD} \backslash x^{1}_{1}}$ means that
$dx^1_1$ is removed from $d{x_{\cD}}$. 
We  integrate the spatial variables iteratively   to find
\begin{align}\label{eq.LMB_00}
	\int \prod_{l=1}^{p}\prod_{j=1}^{m_{p}} G_{t^{l}_{j+1}-t^{l}_{j}}\lt(x^{l}_{j+1},x^{l}_{j}\rt)&\Lambda\Blc x_{\overline{\sV}(\cD)}-x_{\underline{\sV}(\cD)}\Brc dx_{\cD} \nonumber \\
	&\gtrsim\, \ep^{-m\lambda} \prod_{l=1}^{p}\prod_{j=1}^{m_{p}} \lk t^{l}_{j+1}-t^{l}_{j} \rk^{\sfa} 
\end{align} 
for all  	
\[
	t_\cD\in \tilde \Omega_\ep:= \cap_{l=1}^{p}\cap_{j=1}^{m_p}\{t^{l}_{j}\in  I_j  \}\,. 
	\] 
From this inequality,  Remark \ref{Rmk.perm}, and \eqref{eq.cdt_2} 
  we can bound  \eqref{eq.LMB_0}  from below by
\begin{align}
F_\cD & (f_{n_1},  \cdots, f_{n_p}) \nonumber\\
&\ge 
	 \, \ep^{-m\lambda}t^{-m\gamma} \cdot\int   \prod_{l=1}^{p}\prod_{j=1}^{m_{p}} \lk t^{l}_{j+1}-t^{l}_{j} \rk^{\alpha} \1_{I_j}(t^{l}_{j}) \1_{\{0<t^{l}_{1}<\cdots<t^{l}_{m_{p}}<t\}} dt_{\cD} 
	 \nonumber\\
	&\gtrsim \, \ep^{-m\lambda}t^{-m\gamma} \cdot \lc\frac{t}{4m_p}\rc^{2m\sfa}\int \prod_{l=1}^{p}\prod_{j=1}^{m_{p}} \1_{I_j}(t^{l}_{j}) \1_{\{0<t^{l}_{1}<\cdots<t^{l}_{m_{p}}<t\}} dt_{\cD}
	\nonumber \\
	&=: \, \ep^{-m\lambda}t^{-m\gamma} \cdot \lc\frac{t}{4m_p}\rc^{2m\sfa} I_{\ep, p, m} \,,  
	\label{e.6.7} 
\end{align}
 where $I_{\ep, p, m}$ denotes the above multiple integral with respect to $dt_{\cD}$.  
Now let us deal with this integral  $I_{\ep, p, m}$. 
It is easy to see
\begin{align*}
	 I_{\ep, p, m}&= \lk \int \prod_{j=1}^{m_{p}}\1_{I_j}(t_{j}) dt_{1}\cdots dt_{m_p} \rk^{p}
	=\lc\frac{L}{2} \rc^{m_p\times p}\simeq \lc \frac{t}{m_p}\rc^{2m}\,.
\end{align*}
%
Let  $\bD(m_p)$ denote   ${\bD}(f_{m_p},\cdots,f_{m_p})$.
Substituting this bound into \eqref{e.6.7} we have for $\cD\in  {\bD}(m_p)$, 
\begin{align*} 
	 F_{\cD}( {f_{m_p}},&\cdots, {f_{m_p}})\nonumber \\
	\gtrsim\,& \ep^{-m\lambda}t^{-m\gamma} !\cdot \lc\frac{t}{4m_p}\rc^{2m\alpha}\int \prod_{l=1}^{p}\prod_{j=1}^{m_{p}} \1_{I_j}(t^{l}_{j}) \1_{\{0<t^{l}_{1}<\cdots<t^{l}_{m_{p}}<t\}} dt_{\cD}\nonumber \\
	\gtrsim\,& \ep^{-m\lambda}t^{-m\gamma} \cdot \lc\frac{tp}{m} \rc^{2m(\sfa
	+1)}\,. 
\end{align*} 
Since there are $m!$ elements in $\bD(m_p)$,  we have
\begin{align}\label{eq.LMB_000}
\sum_{\cD\in  {\bD}(m_p )}	 F_{\cD} ( {f_{m_p}},\cdots, {f_{m_p}}) 
	\gtrsim\,& m! \ep^{-m\lambda}t^{-m\gamma} \cdot \lc\frac{tp}{m} \rc^{2m(\sfa+1)}\,. 
\end{align}

\textbf{Step 3: }In this   step, we   obtain  the asymptotic behaviors of the term appearing in \eqref{eq.LMB_000} when $m$ is sufficient large. According to Stirling's formula $m!\simeq \sqrt{2\pi m}\cdot(\frac{m}{e})^m$, we arrive at
\begin{align}\label{eq.LMB_1}
	\sum_{\cD\in  {\bD}(m_p )}& F_{\cD}( {f_{m_p}},\cdots, {f_{m_p}})\nonumber \\
	\gtrsim&\, \ep^{-m\lambda}t^{-m\gamma}\cdot \frac{(t\cdot p)^{2m(\sfa+1)}}{m^{m(2\sfa +1)}}
	\simeq\, \lc\ep^{-\lambda}\times \frac{t^{2(\sfa+1)-\gamma}\cdot p^{2(\sfa+1)}}{ m^{2\sfa+1}} \rc^m \,. 
\end{align}
Let us recall that to obtain the above inequality we assumed that  $t,x$ are sufficiently large and $\sfb$ is sufficiently small.
 Consequently, $m$ is also large enough since it satisfies \eqref{eq.cdt_3}. Now  in \eqref{eq.LMB_1}, we can take the value
	\[
	 m_0(\ep) :=\lk C \ep^{-\lambda}t^{2(\sfa+1)-\gamma}p^{2(\sfa+1)} \rk^{\frac{1}{2\sfa+1}}=C\cdot \ep^{-\frac{\lambda}{2\sfa+1}}t^{1+\frac{1-\gamma}{2\sfa+1}}p^{1+\frac{1}{2\sfa+1}}\,.
	\]
With this choice of $m=m_0(\ep)$, the  condition \eqref{eq.cdt_3} i.e. $m\geq \frac{p\cdot t}{2\ep^\sfb }$ together with \eqref{Cond.G2} (i.e. $\sfb(2\sfa+1)-\lambda>0$) imply that
\begin{align*}
 		&~ \ep^{\frac{\sfb (2\sfa+1)-\lambda}{2\sfa+1}}\gtrsim t^{-\frac{1-\gamma}{2\sfa+1}}p^{-\frac{1}{2\sfa+1}}   \\
		\Longleftrightarrow&~ \ep\gtrsim t^{-\frac{1-\gamma}{\sfb(2\sfa +1)-\lambda}}p^{-\frac{1}{\sfb (2\sfa+1)-\lambda}}=:\ep_{t,p} \,. 
	\end{align*} 
	Thus, putting $\ep=\ep_{t,p}$ and $m=m_0(\ep_{t,p})$ into \eqref{eq.LMB_1} we obtain 
	\begin{align*}
		\sum_{\cD\in  {\bD}(m_p )}& F_{\cD}( {f_{m_p}},\cdots, {f_{m_p}})\gtrsim
		\, \exp\lc C\cdot \ep_{t,p}^{-\frac{\lambda}{2\sfa+1}}t^{1+\frac{1-\gamma}{2\sfa+1}}p^{1+\frac{1}{2\sfa+1}} \rc \\
		&=\, \exp\lc C\cdot t^{\frac{1-\gamma}{\sfb (2\sfa+1)-\lambda}\cdot \frac{\lambda}{2\sfa+1}}p^{\frac{1}{\sfb (2\sfa+1)-\lambda}\cdot \frac{\lambda}{2\sfa+1}}\times t^{1+\frac{1-\gamma}{2\sfa+1}}p^{1+\frac{1}{2\sfa+1}} \rc\,,
	\end{align*}
	where
	\begin{align*}
		&1+\frac{1-\gamma}{2\sfa+1}+\frac{1-\gamma}{\sfb (2\sfa+1)-\lambda}\cdot \frac{\lambda}{2\sfa+1}=1+\frac{\sfb\cdot(1-\gamma)}{\sfb (2\sfa+1)-\lambda} 
	\end{align*}
	and
	\begin{align*}
		&1+\frac{1}{2\sfa+1}+\frac{1}{\sfb (2\sfa+1)-\lambda}\cdot \frac{\lambda}{2\sfa+1}=1+\frac{\sfb}{\sfb (2\sfa+1)-\lambda}\,.
	\end{align*}
	This is 
	\begin{align}\label{eq.LMB_2}
	\sum_{\cD\in  {\bD}(m_p )} F_{\cD}( {f_{m_p}},\cdots, {f_{m_p}})\gtrsim\, \exp\lc C\cdot t^{1+\frac{\sfb \cdot(1-\gamma)}{\sfb (2\sfa+1)-\lambda}}p^{1+\frac{\sfb}{\sfb (2\sfa+1)-\lambda}} \rc\,.
	\end{align}
	
As a result, from \eqref{eq.WCE_Prod}, \eqref{eq.LMB_000} and \eqref{eq.LMB_2}  we obtain that
\begin{align*}
	\EE\Blk \prod_{j=1}^{p}u(t,x_j)\Brk
	=\,&\sum_{m=0}^{\infty}\sum_{m_1+\cdots+m_p=2m} \sum_{\cD\in  {\bD}(f_{m_1},\cdots,f_{m_p})} F_{\cD}( {f_{m_1}},\cdots, {f_{m_p}}) \nonumber\\
	\gtrsim\,& \sum_{p\cdot m_p=2m_0}\sum_{\cD\in  \bD(m_p)}F_{\cD}( {f_{m_p}},\cdots, {f_{m_p}})\\
	\gtrsim\,& \exp\lc  t^{1+\frac{\sfb \cdot(1-\gamma)}{\sfb (2\sfa+1)-\lambda}}\cdot p^{1+\frac{\sfb}{\sfb (2\sfa+1)-\lambda}} \rc\,.
\end{align*}
We have completed the proof of Theorem \ref{thm.LMB}.
\end{proof}

\section{Some important SPDEs}\label{s.7} 
In this section, we shall explain the \emph{positivity property} \ref{A1},
 the   \emph{ small ball nondegeneracy property} (\emph{B($\sfa, \sfb$)})   \ref{A2} and the \emph{HLS  total weighted  mass property}   \ref{A3} for some important stochastic PDEs: SHE, $\alpha$-SHE, SWE and SFD. 

\subsection{Stochastic heat equation (SHE)} Firstly, we consider  the well known stochastic heat equation   that has been extensively studied in literature, see \cite{Hu19}  and the references therein. The equation has the following form. 
\begin{equation}\label{eq.SHE}
(\text{SHE})\quad \begin{cases}
  \frac{\partial u(t,x)}{\partial t}=\frac 12 \Delta u(t,x)+u(t,x)\dot{W}(t,x),\quad   t>0,\quad x\in\RR^d\,, \\
  u(0,x)=u_0(x)\,. 
\end{cases}
\end{equation}
In this case the partial differential operator  in the setting of equation \eqref{eq.LDiffu} 
is 
\[
\sL u(t,x) = \frac{\partial u(t,x)}{\partial t}-\frac 12 \Delta u(t,x)\,. 
\]
There is only one initial condition $u(0, x)=u_0(x)$. 
The Green's function and its Fourier transform in spatial variable 
are  respectively:  
\begin{equation}\label{HeatKerG}
  G^{\uph}_t(x)=\frac {1}{(2\pi t)^{d/2}} \exp\lc-\frac{|x|^2}{2t}\rc \quad  \text{and}\quad \cF [G^{\uph}_t(\cdot)](\xi)=\exp\lc-\frac{t|\xi|^2}{2}\rc\,.
\end{equation} 
It is clear that $G_t^{\uph}(x)\ge 0$ is a positive kernel. 
So, the assumption \ref{A1} is obviously satisfied. 
%
We shall show the   \emph{small ball nondegeneracy property} (\emph{B($\alpha, \beta$)})   \ref{A2} and the \emph{HLS  mass property} \emph{M$(\upmu ,\upnu )$}   \ref{A3} in the following proposition \ref{Prop.HSB_Heat} and proposition \ref{Prop.UpperM_SHE} respectively.

\begin{prop}[\textbf{Small Ball Nondegeneracy Property and Lower Moments for SHE}]\label{Prop.HSB_Heat} 
	For the heat kernel $G_t^{\uph}(x)$,  the small ball nondegeneracy  \emph{B(0,2)} holds.
	In fact we have the following statements: 
	\begin{enumerate}[label=(\roman*)]
	\item For all $d\in \bN$, there exist some strict positive constants $C_1$ and $C_2$ independent of
  $t$, $x$ and $\ep$ such that
		\begin{equation}\label{eq.SB_H}
		\inf_{y\in B_{\ep}(x)}\int_{B_{\ep}(x)}G^{\uph }_{t}(y-z)dz \geq C_1  \exp\lc -C_2\frac{t}{\ep^2}\rc\,.
		\end{equation}
	\item Consequently, \emph{B(0,2)} holds for $G_t^{\uph}$, i.e.  there exist a strict positive constant $C$ independent of $t$, $x$ and $\ep$ so that
	\begin{equation}\label{eq.BP_H.2}
		\inf_{y\in B_{\ep}(x)}\int_{B_{\ep}(x)}G^{\uph}_{t}(y-z)dz \geq C\,,
	\end{equation}
	for $0<t\leq \ep^2$.
	\end{enumerate}
	
As a result, assuming $\gamma(\cdot)$ (with $\gamma=2-2H$) and $\Lambda(\cdot)$ satisfy the same conditions   of  Theorem \ref{thm.LMB}, there are some positive constants $c_1$ and $c_2$ independent of $t$, $p$ and $x$ such that 
	\[
	 \EE[|u^{\uph}(t,x)|^p]\geq c_1 \exp\lc c_2\cdot t^{\frac{4H-\lambda}{2-\lambda}}p^{\frac{4-\lambda}{2-\lambda}} \rc\,.
	\]
\end{prop}

\begin{proof}   
We only need to prove   \eqref{eq.SB_H},   which  is related to what is known as small ball property of Brownian motion. The readers can find the related result in immense literatures, for example (5.6.20) in \cite{H2017} for one dimension. We divide the proof into two steps. 

\noindent	\textbf{Step 1}:  Clearly,  we may assume $x=(0,\cdots,0)$. It may be possible to work on the integral directly. However, we feel easier to use the spherical coordinate for the computation of the integral. We 
  employ the following $d$-dimensional  spherical coordinate 
$(z_1,\cdots,z_d)=\Phi(r,\theta, \phi_1,\cdots,\phi_{d-2})$:
	\begin{equation}\label{eq.Spherical}
		\begin{cases}
		z_1=  r\cdot \cos(\phi_1)\\
		z_2=  r\cdot \sin(\phi_1)\cos(\phi_2)\\
		&\cdots \\
		z_{d-2}=  r\cdot \sin(\phi_1)\cdots \sin(\phi_{d-3})\cos(\phi_{d-2}) \\
		z_{d-1}=  r\cdot \sin(\phi_1)\cdots \sin(\phi_{d-2})\cos(\theta) \\
		z_{d-1}=  r\cdot \sin(\phi_1)\cdots \sin(\phi_{d-2})\sin(\theta)\,,
		\end{cases}
	\end{equation}
	where $0\leq\phi_{n}<\pi$, $n=1,\cdots,d-2$, $0\leq\theta\leq2\pi$. 
	The Jacobian determinant of this transformation 
	is 
	\[
	J_d=r^{d-1}\prod_{k-1}^{d-2}\sin^{d-1-k}(\phi_k)\,.
	\]
Since $G^h_t(\cdot)$ is rotation invariant as a function in $\RR^d$
we only need to consider $y=(r_0,0,\cdots,0)$ 
	 for some fixed $r_0\in(0,\ep)$. Set $B_{\ep}(r_0):=B_{\ep}(y)$, therefore
	\begin{align}
		\int_{B_{\ep}(0)} G^{\uph}_{t}(y-z)dz 
		\geq&\ \int_{B_{\ep}(r_0)\cap B_{\ep}(0)} \frac{1}{(2\pi t)^{d/2}}\exp\lc -\frac{|z|^2}{2t} \rc dz \nonumber\\
		\simeq&\ \int_{0}^{\ep}\int_{[0,\pi)^{d-2}}\int_{0}^{2\pi} \frac{1}{(2\pi t)^{d/2}}\exp\lc -\frac{r^2}{2t} \rc \label{eq.H_L.sc}\\
		 &\qquad\qquad\qquad\quad \times \1_{B_{\ep}(r_0)}(\Psi(r,\theta,\phi))\cdot |J_d| d\theta d\phi dr\,. \nonumber		
	\end{align} 
Notice that the identity  
\[
\1_{B_{\ep}(r_0)}(\Psi(r,\theta,\phi))=\1_{B_{\ep}(0)}((r_0,0,\dots,0)-\Psi(r,\theta,\phi))
\]
 can be expressed as
	\begin{equation}\label{Event_B}
		\begin{split}
		\{(r,\theta,\phi)&\in[0,\ep]\times[0,2\pi)\times[0,\pi)^{d-2}:
	  	r^2\sin^2(\phi_1)+[r\cos(\phi_1)-r_0]^2\leq \ep^2 \}\\
		 &=\{(r,\theta,\phi)\in[0,\ep]\times[0,2\pi)\times[0,\pi)^{d-2}:r^2+r_0^2-2r\cdot r_0 \cos(\phi_1)\leq \ep^2 \} \,.		 
		\end{split}
	\end{equation}
	In order to estimate the lower bound of \eqref{eq.H_L.sc}, we need the following particular subset of $\{(r,\theta,\phi)\in[0,\ep]\times[0,2\pi)\times[0,\pi)^{d-2}: \Psi(r,\theta,\phi)\in B_{\ep}(r_0) \}$:
	\begin{equation}\label{Event_B.sub}
		\begin{split}
		&S_{\ep}(r,\theta,\phi):= \{(r,\theta,\phi)\in[0,\ep]\times[0,2\pi)\times[0,\pi/3)\times[0,\pi/2)^{d-3}: \\
		 &\qquad\qquad\qquad\qquad\qquad\qquad\qquad\qquad\quad r^2+r_0^2-2r\cdot r_0 \cos(\phi_1) \leq \ep^2 \}\,.
		\end{split}
	\end{equation}
	
	Because for $\phi_1\in [0,\pi/3)$, we always have
	\[
	 r^2+r_0^2-2r r_0 \cos(\phi_1) \leq r^2+r_0^2 -rr_0  \leq \ep^2\,,
	\]
	if $0\leq r, r_0\leq \ep$. On the domain 
	 $S_{\ep}(r,\theta,\phi)$, the indicate 
	 function  $\1_{S_{\ep}}(r,\theta,\phi)
	 :=\1_{S_{\ep}(r,\theta,\phi)}(r,\theta,\phi)=1$. 
	 Then we have from \eqref{eq.H_L.sc}
	\begin{align}
		\int_{B_{\ep}(0)} &G^{\uph}_{t}(y-z)dz \nonumber \\
		\gtrsim& \int_{0}^{\ep}\int_{[0,\pi)^{d-2}}\int_{0}^{2\pi} \frac{1}{(2\pi t)^{d/2}}\exp\lc -\frac{r^2}{2t} \rc\times \1_{S_{\ep}}(r,\theta,\phi)\cdot |J_d| d\theta d\phi dr \nonumber\\
		\gtrsim& \int_{0}^{\ep} \frac{1}{(2\pi t)^{d/2}}\exp\lc -\frac{r^2}{2t} \rc \times r^{d-1} dr \simeq  \int_{0}^{\lc\frac{\ep}{\sqrt{t}}\rc^{d}}   \exp\lc -\frac{\tilde{r}^{\frac{2}{d}}}{2} \rc d\tilde{r}\,,\label{eq.H_Lclaim}
	\end{align}
	where we have used the change of varible $r\to \tilde{r}=r/\sqrt{t}$ in the last line. 
	
	\textbf{Step 2}:   We shall prove \eqref{eq.H_Lclaim} is greater than $C_1  \exp\lc -\frac{C_2 \cdot t}{\ep^2}\rc$ by showing the following  {claim}. For fixed $\nu>0$, one can find a constant $c_{\nu}$ such that $c_{\nu}\cdot\int_{0}^{\infty} \exp\lc -\frac{1}{2}r^{\nu}\rc dr=1$. We  {claim} that there exists a constant $c>\frac{(\nu+1)^2}{4\nu}$ such that $\forall~\delta:=\lc\frac{\ep}{\sqrt{t}}\rc^{d}>0$, 
	\begin{equation}\label{eq.H_Claim}
		\int_{0}^{\delta} e^{-\frac{r^{\nu}}{2}}dr\geq c^{-1}_{\nu}\cdot e^{-\frac{c}{\delta^{\nu}}}\,.
	\end{equation}
	This is equivalent to prove
	\begin{equation*}
		c_{\nu}\cdot\int_{\delta}^{\infty} \exp\lc -\frac{1}{2}r^{\nu}\rc dr+e^{-\frac{c}{\delta^{\nu}}}\leq 1\,.
	\end{equation*}
	Let 
	\[
	g(\delta)=c_{\nu}\cdot\int_{\delta}^{\infty} e^{-\frac{r^{\nu}}{2}}dr+e^{-\frac{c}{\delta^{\nu}}}\,.
	\]
	 It is easy to see that $g(\delta)$ is continuous and $g(0)=g(\infty)=1$. So in order to prove $g(\delta)\leq 1$ for all $\delta>0$, it suffices to show that if  $c>\frac{(\nu+1)^2}{4\nu}$, then
	\[
	 g'(\delta)=\frac{\nu\cdot c}{\delta^{\nu+1}}e^{-\frac{c}{\delta^{\nu}}}-c_{\nu}\cdot e^{-\frac{\delta^{\nu}}{2}} =0
	\]
	has exactly one root. It is clear  that  this is equivalent to
	\begin{align*}
		\frac{\nu\cdot c}{c_{\nu}} \cdot e^{\frac{\delta^\nu}{2}}=\delta^{\nu+1}e^{\frac{c}{\delta^{\nu}}}\,&\Leftrightarrow \,\exp\lc\frac{c}{\delta^{\nu}}+(\nu+1)\ln(\delta)-\frac{\delta^{\nu}}{2}-\ln\lc  \frac{\nu\cdot c}{c_{\nu}}\rc\rc=1 \\
		&\Leftrightarrow \,h(\delta)=\frac{c}{\delta^\nu}+(\nu+1)\ln(\delta)-\frac{\delta^\nu}{2}-\ln\lc  \frac{\nu\cdot c}{c_{\nu}}\rc=0
	\end{align*}
	has exactly one root. One can notice that $h(0+)=+\infty$ and $h(+\infty)=-\infty$. Then $h(\ep)$ has at least one root. Next, we shall show it has at most one root, which suffices to argue that
	\[
	 h'(\delta)=-\frac{1}{\delta^{\nu+1}}\lk\lc\delta^{\nu}-\frac{\nu+1}{2}\rc^2+\nu\cdot c-\frac{(\nu+1)^2}{4} \rk=0
	\]
	has no root for $\delta>0$. But this is verified when $c>\frac{(\nu+1)^2}{4\nu}$. Lastly, the fact $g'(\delta)=0$ has only one root and the intermediate value theorem imply that the claim 
	 \eqref{eq.H_Claim} holds.
	 
	 Letting $\nu=2/d$ and $\delta=(\frac{\ep}{\sqrt{t}})^{d}$ in \eqref{eq.H_Claim}, we get \eqref{eq.H_Lclaim} is greater than $C_1  \exp\lc -\frac{C_2 \cdot t}{\ep^2}\rc$ for some constant $C_1$ and $C_2$. Thus, we have completed the proof of \eqref{eq.SB_H}. 
\end{proof}

\begin{prop}[\textbf{HLS mass Property  and Upper Moments for SHE}]\label{Prop.UpperM_SHE}
Assume $\gamma(\cdot)$ (with $\gamma=2-2H$) 
and $\Lambda(\cdot)$ with
$\lambda<2$ satisfy the same conditions 
as in Theorem \ref{thm.UMB} or Theorem \ref{thm.UMB2}. 
Then for the heat kernel $G_t^{\uph}(x-y)$, we have \ref{A3} or \ref{A3'} with \emph{$M(-\frac{\lambda}{2})$} hold. In other words, for all $d\in \bN$, there exist some strict positive constants $C_1$, $C_2$ and $C_3$ do not depend on $t$ and $x$ such that
\begin{equation}\label{A.F_SHE}
\sup\limits_{x,x'\in\RR^d} \int_{\RR^{2d}} G^{\uph}_t(x-y)\Lambda(y-y') G^{\uph}_t(x'-y')dydy' \leq C\cdot t^{-\frac{\lambda}{2}}\,,
\end{equation}
or denoting $\mu(d\xi)=\hat{V}(\xi) d\xi$
\begin{equation}\label{A.F_SHE'}
	\sup_{\eta\in\RR^d} \int_{\RR^d}|\hat{G}^{\uph}_t(\xi-\eta)|^2 \mu(d\xi) \leq C_3 \cdot t^{-\frac{\lambda}{2}}\,.
\end{equation}

As a result,  we have the upper $p$-th ($p\geq 2$) moments for $u^{\uph}(t,x)$ for any $d\geq 1$. More precisely,    for some constants $C_1$ and $C_2$ that are independent of $t$, $p$ and $x$ we can get
		 $$\EE[|u^{\uph}(t,x)|^p]\leq C_1\cdot \exp\lc C_2\cdot t^{\frac{4H-\lambda}{2-\lambda}} p^{\frac{4-\lambda}{2-\lambda}} \rc\,.$$
\end{prop}
\begin{proof}  We only need to prove \eqref{A.F_upper_1} with 
\emph{$\bar{M}(0,-\frac{\lambda}{2})$}. This is,
	\begin{align*}
		\sup_{x\in\RR^d}\int_{\RR^d} G^{\uph}_t(x-y)dy=& \int_{\RR^d} G^{\uph}_t(y)dy=1\,, \\
		\sup_{x\in\RR^d} \int_{\RR^d} G^{\uph}_t(x-y)\Lambda(y)dy \,\ls&\,  
		 \sup_{x\in\RR^d} \EE|\sqrt{t}X-x|^{-\la}    \leq C\cdot t^{-\frac{\lambda}{2}}\,,
	\end{align*}
	where $X$ is a standard normal random variable and
	 the above last inequality follows from \cite[Lemma A.1]{HNS11}. 
	 
	 For the \eqref{A.F_SHE'}, it is easy to
	\begin{align*}
		\sup_{\eta\in\RR^d} \int_{\RR^d}|\hat{G}^{\uph}_t(\xi-\eta)|^2\mu(d\xi)=&\sup_{\eta\in\RR^d} \int_{\RR^d}e^{-t|\xi-\eta|^2}\mu(d\xi)  \\
		\leq&\ t^{-\frac{\lambda}{2}}\cdot \sup_{\eta\in\RR^d} \int_{\RR^d}\frac{|\xi|^{\lambda-d}}{1+|\xi-\eta|^2}d\xi  \leq C \cdot t^{\hbar}\,.
	\end{align*}
	So, we obtain the upper moment bound. 
\end{proof}


\subsection{Fractional spatial   equations: Space nonhomogeneous case} 
The next  model is the generalized $d~(\geq 1)$-spatial dimensional 
fractional stochastic $\alpha$-heat equation ($\alpha$-SHE) that has been considered in \cite{BC2014,BC2016,CHW2018}: 
\begin{equation}\label{eq.SHE_al}
(\text{$\alpha$-SHE})\quad \begin{cases}
  \frac{\partial u(t,x)}{\partial t}=-(-\nabla(a(x)\nabla))^{\alpha/2}u(t,x)+u(t,x)\dot{W}(t,x),\quad   t>0,\quad x\in\RR^d\,, \\
  u(0,x)=u_0(x)\,,
\end{cases}
\end{equation}
where $0<\alpha<2$,  $a(\cdot):\RR^d\to \RR^{d^2}$ is a matrix valued function whose entries are 
H\"older continuous, and
 there exists a constant $c\geq 1$ such that $c^{-1}\cdot Id\leq a(x)\leq c\cdot Id$. 
 The operator $\sL$ is
 \[
 \sL u(t,x)=\frac{\partial u(t,x)}{\partial t}+
 (-\nabla(a(x)\nabla))^{\alpha/2}u(t,x)
 \]
  and the corresponding Green's function  $G^{\uph,\alpha}_t(x)$
satisfies the following Nash's H\"{o}lder estimates (see e.g. \cite{CHW2018} for more details):
	\begin{equation}\label{eq.Nash}
		\frac 1C \lc t^{-\frac{d}{\alpha}}\wedge \frac{t}{|x-y|^{d+\alpha}} \rc\leq G^{\uph,\alpha}_{t}(x,y)\leq C\lc t^{-\frac{d}{\alpha}}\wedge \frac{t}{|x-y|^{d+\alpha}} \rc\,,
	\end{equation}
	and $ I_0(t,x) =G^{\uph,\alpha}_t\ast u_0(x)$.  
%
Clearly,  \eqref{eq.Nash} ensures the \emph{positivity} of $G^{\uph,\alpha}_{t}(x)$ when $\alpha\in(0,2)$. We still need to take care of the \emph{small ball nondegeneracy property} \ref{A2} with \emph{B($\alpha, \beta$)} and the \emph{HLS  mass property} \ref{A3} with \emph{M$(0,-\frac{\lambda}{\alpha})$}. 

\begin{prop}[\textbf{Small Ball Nondegeneracy Property and Lower Moments for $\alpha$-SHE}]\label{Prop.HSB_Heat_al} 
	For the heat kernel $G_t^{\uph,\alpha}(x)$, we have \emph{B(0,$\alpha$)} holds:
	\begin{enumerate}[label=(\roman*)]
	\item For $\alpha\in(0,2)$ and $d\in \bN$, there exist 
	some strict positive constants $C_1$ and $C_2$ do not depend on $t$ and $\ep$ such that
		\begin{equation}\label{eq.SB_H_al}
		\inf_{y\in B_{\ep}(x)} 
		\int_{B_{\ep}(x)}G^{\uph,\alpha}_{t}(y,z)dz \geq C_1  \exp\lc -C_2\frac{t}{\ep^\alpha}\rc\,.
		\end{equation}
	\item Consequently, \emph{B(0,$\alpha$)} holds for $G_t^{\uph,\alpha}$, i.e. there exist a strict positive constant $C$ independent of $t$ and $\ep$ so that
	\begin{equation}\label{eq.BP_H.2}
		\inf_{y\in B_{\ep}(x)} \int_{B_{\ep}(x)}G^{\uph,\alpha}_{t}(y,z)dz \geq C\,,
	\end{equation}
	for $0<t\leq \ep^\alpha$.	
	\end{enumerate}
	
As a result, assuming $\gamma(\cdot)$ (with $\gamma=2-2H$) and $\Lambda(\cdot)$ satisfy the same conditions of  
Theorem \ref{thm.LMB}, we have 
 the lower $p$-th ($p\geq 2$) moment bound:  there are 
 constants $c_1$ and $c_2$ independent of $t$, $p$ and $x$  such that 
	\[
	 \EE[|u^{\uph,\alpha}(t,x)|^p]\geq c_1 \exp\lc c_2\cdot t^{\frac{2\alpha H-\lambda}{\alpha-\lambda}}p^{\frac{2\alpha-\lambda}{\alpha-\lambda}} \rc\,.
	\]
\end{prop}
\begin{proof}
	The proof is similar to the SHE case except  now we   have   the Nash's H\"older estimates \eqref{eq.Nash} instead of the 
	the precise form  of $G^{\uph,\alpha}_{t}(x)$.
	
		By lower bound in  the Nash's  inequality  \eqref{eq.Nash}, we have
	\begin{align}\label{eq.Nash_L}
	 G^{\uph,\alpha}_{t}(x,y) \gtrsim&\ t^{-\frac{d}{\alpha}}\exp\lc -C_{\alpha,d}\cdot\frac{|x-y|^{\alpha}}{t} \rc \,,
	\end{align}
	since $1\wedge |x|^{-1}\geq C_{1,\alpha}\cdot\exp\lc-C_{2,d}\cdot|x|^{\alpha}\rc$ for $\alpha>0$.	
Thus  \eqref{eq.BP_H.2}  can be proved the same way as that of \eqref{eq.SB_H_al}. 
\end{proof}

\begin{prop}[\textbf{HLS  mass Property and Upper Moments for $\alpha$-SHE}]\label{Prop.UpperM_SHE_al}
	Assume $\gamma(\cdot)$ (with $\gamma=2-2H$) and $\Lambda(\cdot)$ with
	$\lambda<\alpha$ satisfy the same conditions of  Theorem \ref{thm.UMB} or Theorem \ref{thm.UMB2}. 
	Then for the heat kernel $G_t^{\uph,\alpha}(x,y)$, we have \ref{A3}  or \ref{A3'} with \emph{$M(-\frac{\lambda}{\alpha})$} hold. In other words, for all $d\in \bN$, there exist some strict positive constants $C_1$ and $C_2$ independent of   $t$ and $x$ such that
	\begin{equation}\label{A.F_SHE_al}
	\sup\limits_{x,x'\in\RR^d} \int_{\RR^{2d}} G^{\uph,\alpha}_t(x,y)\Lambda(y-y') G^{\uph,\alpha}_t(x',y')dydy' \leq C\cdot t^{-\frac{\lambda}{\alpha}}\,. 
	\end{equation}
Furthermore,  there is a positive kernel 
 $Q_t(x-y)$ such that  $G^{\uph,\alpha}_t(x,y)\leq Q_t(x-y)$ and
	\begin{equation}\label{A.F_SHE_al'}
	\sup_{\eta\in\RR^d} \int_{\RR^d}|\hat{Q}_t(\xi-\eta)|^2 |\mu|(d\xi) \leq C_3 \cdot t^{-\frac{\lambda}{\alpha}} 
	\end{equation}
	with $|\mu|(d\xi)=|\hat{V}(\xi) |d\xi$.
	
	 Consequently, we have the upper $p$-th ($p\geq 2$) moment bounds. 
	 This is,    for some constants $C_1$ and $C_2$ that are independent of $t$, $p$ and $x$ we have 
		 $$\EE[|u^{\uph}(t,x)|^p]\leq C_1\cdot \exp\lc C_2\cdot t^{\frac{2\alpha H-\lambda}{\alpha-\lambda}}p^{\frac{2\alpha-\lambda}{\alpha-\lambda}} \rc\,.$$
\end{prop}

\begin{proof}
Presumably, we may  use \eqref{eq.Nash} 
to obtain the desired bounds.  However, we  will use    Pollard's formula in \cite{CHW2018}
to prove this proposition. 
	\begin{equation}\label{eq.Pollard}
		e^{-u^{\frac{\alpha}{2}}}=\int_{0}^{\infty} e^{-us}g(\alpha/2,s)ds\,,~u\geq 0\,,
	\end{equation}
	where $g(\alpha,s)$ is a probability density function of $s\geq 0$ and defined in (1.2) in \cite{CHW2018}. By Proposition 2.2 there, we have 
	\begin{align}\label{eq.Pollard_Est}
		G^{\uph,\alpha}_{t}(x,y)=\,&\int_{0}^{\infty} p(t^{\frac{2}{\alpha}}s,x,y)g(\alpha/2,s)ds\nonumber \\
		\leq\,& C\int_{0}^{\infty} t^{-\frac{d}{\alpha}}s^{-\frac{d}{2}}\exp\lc-\frac{|x-y|^2}{Ct^{2/\alpha}s}\rc g(\alpha/2,s)ds=:Q_{t}(x-y)\,.
	\end{align}
	Therefore, it is sufficient to show  the assumption \ref{A3} can be archived with  
	\emph{$\bar{M}(0,-\frac{\lambda}{2})$} (i.e. the estimates \eqref{A.F_upper_1}) for $Q_{t}(x-y)$. 
	It is not hard to derive that
	\begin{align*}
		\sup\limits_{x\in\RR^d} \int_{\RR^d}Q_{t}(x-y)dy\ \ls \ \int_{0}^{\infty}g(\alpha/2,s)ds<\infty\,,
	\end{align*}
	and 
	\begin{align*}
		\sup\limits_{x\in\RR^d} \int_{\RR^d}& Q_{t}(x-y)\Lambda(y)dy \\
		\ls& \ \int_{0}^{\infty} t^{-\frac{d}{\alpha}}s^{-\frac{d}{2}}\lk \sup\limits_{x\in\RR^d} \int_{\RR^d}\exp\lc-\frac{|x-y|^2}{Ct^{2/\alpha}s}\rc \Lambda(y)dy\rk g(\alpha/2,s)ds \\
		\ls& \ t^{-\frac{\lambda}{\alpha}}\cdot \int_{0}^{\infty} s^{-\frac{\lambda}{2}}g(\alpha/2,s)ds\leq C_2\cdot t^{-\frac{\lambda}{\alpha}}\,, 
	\end{align*}
	where we have applied rearrangement inequality and \cite[Proposition 2.1]{CHW2018}. 
	
	Moreover, for the Fourier transform of $Q_{t}(x)$  with respect to  $x$, we have
	\begin{align*}
		\cF[Q_{t}(\cdot)](\xi)\simeq\,& \int_{0}^{\infty} \exp\lc-Cs\cdot t^{2/\alpha}|\xi|^2 \rc g(\alpha/2,s)ds\\
		\simeq\,&\exp\lc-\Blk Ct^{2/\alpha}|\xi|^2\Brk^{\alpha/2}\rc\,=\, e^{-C_{\alpha}\cdot t|\xi|^\alpha}\,.
	\end{align*}
	Finally, it is relatively easy to see that the assumption \eqref{A.F_SHE_al'} can be archived. Then the upper moment bound follows.  
\end{proof}

\subsection{Stochastic wave equations}
the lower moment bounds for $d$-dimensional stochastic wave equation (SWE) is   one of the SPDEs  that  motivated 
this study. This type of equations has been well-studied in literature. 
There are several works on the upper bounds for any moments. But the lower
  bounds are only known for the second moments except in a few cases. 
 (see e.g. \cite{BC2016,DM2009}).  We give a more complete results for all moments. 
This equation has the following form (we consider only $d=1,2,3$): 
\begin{equation}\label{eq.SWE}
(\text{SWE})\quad \begin{cases}
  \frac{\partial^2 u(t,x)}{\partial t^2}=\frac{\partial^2 u(t,x)}{\partial x^2}+u(t,x)\dot{W}(t,x),\quad   t>0,\quad x\in\RR^d\,, \\
  u(0,x)=u_0(x)\,,\quad\frac{\partial}{\partial t}u(0,x)=v_0(x)\,.
\end{cases}
\end{equation}
The operator $\sL$ has the form
\[
\sL u(t,x) = \frac{\partial^2 u(t,x)}{\partial t^2}-\frac{\partial^2 u(t,x)}{\partial x^2}\,. 
\]
The associated  Green's function   has different forms for different dimensions.
More precisely, it is given by 
\begin{equation}\label{WaveKerG}
\begin{cases}
	G^{\upw}_t(x)=\frac 12 \1_{\{|x|<t\}}\,, &d=1\,, \\
	G^{\upw}_t(x)=\frac 1{2\pi} \frac{1}{\sqrt{t^2-|x|^2}} \1_{\{|x|<t\}}\,, &d=2\,,\\
	G^{\upw}_t(dx)=\frac 1{4\pi} \frac{\sigma_t(dx)}{t} \,, &d=3\,,
\end{cases}
\end{equation}
where $\sigma_t(dx)$ is a surface measure on the sphere  $\partial B_t(0)
\subseteq \RR^3$ with center  at $0$ and radius $t$, 
with total mass $4\pi t^2$ and $G^{\upw}_t(\RR^3)=t$. 
It is well known that $G^{\upw}_t(\cdot)$ may  not be positive  when $d\geq 4$.
On the other hand for any dimension    $d$, 
the Fourier transform of  $G^{\upw}_t(\cdot)$ has the  same form given by 
\[
 \cF [G^{\upw}_t(\cdot)](\xi)=\frac{\sin(t|\xi|)}{|\xi|}\,,\quad \xi\in\RR^d\,.
\]
In this case we also have $I^{\upw}_0(t,x):=\frac{\partial}{\partial t} G^{\upw}_t\ast u_0(x)+G^{\upw}_t\ast v_0(x)$.

When $d=1,2$, $G^{\upw}_t(x)$ are positive functions and when   $d=3$   it is a positive measure. Thus, the assumption \ref{A1} is satisfied for wave kernel $G^W_t(dx)$. The next two propositions are devoted 
to   \ref{A2} and \ref{A3}.
\begin{prop}[\textbf{Small Ball Nondegeneracy Property and Lower Moments for SWE}]\label{Prop.HSB_Wave}
	For the wave kernel $G_t^{\upw}(x)$ defined by  \eqref{WaveKerG}, we have \emph{B(1,1)} holds:
	\begin{enumerate}[label=(\roman*)]
	\item When $d=1$ and $d=2$, there exist  
	 strict positive constants $C_1$ and $C_2$, independent of 
	  $t$, $\ep$ and $y$ such that
	\begin{equation}\label{eq.SB_W.12}
		\inf_{y\in B_{\ep}(x)} \int_{B_{\ep}(x)}G^{\upw}_{t}(y-z)dz \geq C_1\cdot t \exp\lc -C_2\frac{t}{\ep}\rc \,.
	\end{equation}
	Consequently,  there exist a strict positive constant $C$ independent of $t$, $\ep$ and $y$ so that
	\begin{equation}\label{eq.SBP_W.12}
		\inf_{y\in B_{\ep}(x)} \int_{B_{\ep}(x)}G^{\upw}_{t}(y-z)dz \geq C\cdot t\,,
	\end{equation}
	for $0<t\leq \ep$.
	\item When $d=3$,  there exists a strict positive constant $C$ independent of 
	 $t$, $\ep$ and $y$ such that
	\begin{equation}\label{eq.SBP_W.3}
		\inf_{y\in B_{\ep}(x)} \int_{B_{\ep}(x)}G^{\upw}_{t}(y-dz) \geq C \cdot t\,,
	\end{equation}
	for $0<t\leq \ep$.
	\end{enumerate}  
As a consequence, assuming $\gamma(\cdot)$ (with $\gamma=2-2H$) and $\Lambda(\cdot)$ satisfy the same conditions of  Theorem \ref{thm.LMB}, we have the following lower moment bounds for the solution:
	\[
	 \EE[|u^{\upw}(t,x)|^p]\geq c_1 \exp\lc c_2\cdot t^{\frac{2H+2-\lambda}{3-\lambda}}\cdot p^{\frac{4-\lambda}{3-\lambda}} \rc 
	\]
	for some constants $c_1$ and $c_2$ independent of $t$, $p$ and $x$. 
\end{prop}
\begin{rmk}
	The \emph{small ball nondegeneracy property} of wave kernel $G_t^{\upw}$ is motivated by the following fact when $d=1$. Let us illustrate it with $x=y=0$. Then the left hand of \eqref{eq.SB_W.12} can be evaluated exactly as
	\[
	 \int_{-\ep}^{\ep} G^{\upw}_{t}(z)dz=\int_{-\ep}^{\ep}\frac 12 \1_{\{|z|<t\}}dz=t\wedge \ep\,.
	\]
	And then it is not hard to see
	\[
	 t\wedge \ep=\ep\cdot \lc\frac{t}{\ep}\wedge 1\rc \geq \ep\cdot\lc C_1 \cdot  \frac {t}{\ep} \exp\lc -C_2\frac{t}{\ep}\rc\rc=C_1\cdot t \exp\lc -C_2\frac{t}{\ep}\rc\,,
	\]
	which is the right hand of \eqref{eq.SB_W.12}.
\end{rmk}

\begin{proof} We shall give the proof of Proposition \ref{Prop.HSB_Wave} for $d=1,2,3$ in  three steps seperately.

	\textbf{Step 1 ($d=1$)}: It is clear that we only need to show \eqref{eq.SB_W.12}. Without loss of generality, we may assume $x=0$. Let us consider $d=1$ at first. Because $G^{\upw}_{t}(y-z)=\frac 12 \1_{\{|y-z|<t\}}$, then \eqref{eq.SBP_W.12} becomes
	\begin{align}\label{eq.SBP_W1}
		\int_{\RR} G^{\upw}_{t}(y-&z)G^{\upw}_{\ep}(z)dz \nonumber\\
		\simeq\,&\int_{\RR}\cF[ G^{\upw}_{t}(y-\cdot)](\xi)\cF[ G^{\upw}_{\ep}(\cdot)](\xi)d\xi \nonumber\\
		\simeq\,&\int_{\RR} e^{-\iota y\xi}\frac{\sin(t|\xi|)}{|\xi|}\frac{\sin(\ep|\xi|)}{|\xi|}d\xi \nonumber\\
		\simeq\,&\int_{\RR} e^{-\iota y\xi}|\xi|^{-2}\lk \sin^2\lc \frac 12 |t+\ep||\xi|\rc-\sin^2\lc \frac 12 |t-\ep||\xi|\rc \rk d\xi \nonumber\\
		\simeq\,&(|t+\ep|- y)\1_{\{|y|<|t+\ep|\}}-(|t-\ep|- y)\1_{\{|y|<|t-\ep|\}} \,,
	\end{align} 
	where in the last line we have applied the Fourier transform (e.g. 17.34(21) in \cite{GR2014})
	\[
	 \cF[x^{-2}\sin^2(ax)](\xi)=\cF_c[x^{-2}\sin^2(ax)](\xi)= \frac{\pi}{2}(a-\xi/2)\1_{\{\xi<2a\}}\,.
	\]
	
	The rest is routine. We split \eqref{eq.SBP_W1} into two cases: $t>\ep$ and $t\leq \ep$. Noticing $|y|\leq \ep$, when $t>\ep$ we can bound  \eqref{eq.SBP_W1} below by
	\[
	 ((t+\ep)- y)-((t-\ep)- y)\1_{\{|y|<t-\ep\}}\geq 2\ep \1_{\{|y|<t-\ep\}}+t\1_{\{|y|\geq t-\ep\}}\geq \ep\,.
	\]
	The case $t\leq \ep$ can be done similarly, so we omit the details. Therefore, we obtain
	\[
	\int_{B_{\ep}(x)}G^{\upw}_{t}(y-z)dz \geq t\wedge \ep \geq C_1\cdot t \exp\lc -C_2\frac{t}{\ep}\rc \,.
	\]
	We have completed the proof of \eqref{eq.SBP_W.12} when $d=1$.
	
	\textbf{Step 2 ($d=2$)}: Recall that $G^{\upw}_{t}(y-z)=\frac 1{2\pi} \frac{1}{\sqrt{t^2-|y-z|^2}} \1_{\{|y-z|<t\}}$.  Then
	\begin{equation}\label{eq.SBP_W2}
		\begin{split}
			\int_{\RR^2} G^{\upw}_{t}(y-z)&\1_{B_\ep}(z)dz \\
		\gtrsim& \int_{\RR^2} \frac 1t \1_{\{|y-z|<t\}}\1_{\{|z|<\ep\}}dz \\
		\simeq&\frac 1t\int_{\RR^2} \1_{\{|y_1-z_1|<t\}}\1_{\{|y_2-z_2|<t\}}\1_{\{|z_1|<\ep\}}\1_{\{|z_2|<\ep\}}dz \\
		\simeq&\frac 1t\lc\int_{\RR}\1_{\{|y-z|<t\}}\1_{\{|z|<\ep\}}dz \rc^2 \\
		\gtrsim&\frac 1t \lc C_1\cdot t \exp\lc -C_2\frac{t}{\ep}\rc \rc^2=C_1\cdot t \exp\lc -C_2\frac{t}{\ep}\rc \,,
		\end{split}
	\end{equation}
	where we have applied the result in $d=1$ to derive the inequality last line in \eqref{eq.SBP_W2}. Thus, the proof of \eqref{eq.SBP_W.12} when $d=2$ has been completed.
	
	\textbf{Step 3 ($d=3$)}: Let us recall that now $G^{\upw}_{t}(dz)=\frac 1{4\pi} \frac{\sigma_t(dz)}{t}$ where $\sigma_t(dz)$ is the  surface measure on $\partial B_t(0)$. We may assume $x=0$ and simplify $B_{\ep}(0)$ as $B_{\ep}$.  Then \eqref{eq.SBP_W.3} becomes
	\begin{align}\label{eq.SBP_W3}
		\int_{\RR^3} \1_{B_{\ep}}(z)& G^{\upw}_{t}(y-dz) \nonumber \\
		&=\frac 1{4\pi t}\int_{\partial B_{t}} \1_{B_{\ep}}(y-z) \sigma_t(dz)\nonumber \\
		&=\frac 1{4\pi t}\int_{0}^{2\pi}\int_{0}^{\pi} \1_{B_{\ep}}(y-\Psi(\theta,\phi))\lt\|\frac{\partial \Psi}{\partial \theta}\times \frac{\partial \Psi}{\partial \phi} \rt\| d\phi d\theta\nonumber \\
		&=\frac {t}{4\pi}\int_{0}^{2\pi}\int_{0}^{\pi} \1_{B_{\ep}}(y-\Psi(\theta,\phi))|\sin(\phi)| d\phi d\theta\,,
	\end{align}
	where the parametrization is the three dimensional spherical coordinate (i.e. $d=3$ in \eqref{eq.Spherical}):
	\[
	 \Psi(\theta,\phi)=(z_1(\theta,\phi),z_2(\theta,\phi),z_3(\theta,\phi))=(t\sin(\phi)\cos(\theta),t\sin(\phi)\sin(\theta),t\cos(\phi))\,.
	\]
	
	Similarly, we can select the particular subset as in  \eqref{Event_B.sub} so that we can bound \eqref{eq.SBP_W3} below as
	\begin{align*}
		\frac {t}{4\pi}\int_{0}^{2\pi}\int_{0}^{\pi} \1_{B_{\ep}}(y-\Psi(\theta,\phi))|\sin(\phi)| d\phi d\theta \geq \frac {t}{4\pi}\int_{0}^{2\pi}\int_{0}^{\pi/3} |\sin(\phi)| d\phi d\theta=t/4\,.
	\end{align*}
	As a result, we have completed the proof of \eqref{eq.SBP_W.3}.	
\end{proof}

\begin{prop}[\textbf{HLS mass Property  and  Upper Moments for SWE}]\label{Prop.UpperM_SWE}
	Assume $d= 1,2,3$, $\gamma(\cdot)$ (with $\gamma=2-2H$) and $\Lambda(\cdot)$ with $\lambda<2 \wedge d$ satisfy the same conditions of  Theorem \ref{thm.UMB} or Theorem \ref{thm.UMB2}.
	 Then for the wave kernel $G^{\upw}_{t}(x)$, we have \ref{A3} with \emph{M$(2-\lambda)$} or \ref{A3'} with \emph{$M(2-\lambda)$} hold. In other words, for $d=1,2,3$, there exists some strict positive constants $C $  
	 independent of   $t$ and $x$ such that
	\begin{equation}\label{A.F_SWE}
		\sup\limits_{x,x'\in\RR^d} \int_{\RR^{2d}} G^{\upw}_t(x-y)\Lambda(y-y') G^{\upw}_t(x'-y')dydy' \leq C\cdot t^{2-\lambda}\,,
	\end{equation}
Denoting $\mu(d\xi)=\hat{V}(\xi) d\xi$
	\begin{equation}\label{A.F_SWE'}
	\sup_{\eta\in\RR^d} \int_{\RR^d}|\hat{G}^{\upw}_t(\xi-\eta)|^2 |\mu|(d\xi) \leq C  \cdot t^{2-\lambda}\,.
	\end{equation}
	
	Consequently, we have the desired  upper $p$-th ($p\geq 2$) moment bounds for 
	  the solution $u^{\upw}(t,x)$ when $d= 1,2,3$. This is, we can find 
 constants $C_1$ and $C_2$ that are independent of $t$, $p$ and $x$ such that 
		 $$\EE[|u^{\upw}(t,x)|^p]\leq C_1\cdot \exp\lc C_2\cdot t^{\frac{2H+2-\lambda}{3-\lambda}}\cdot p^{\frac{4-\lambda}{3-\lambda}} \rc\,.$$
\end{prop}
\begin{proof}
	It is clear we only need to show  \ref{A3} holds for $G^{\upw}_{t}(x)$ with \emph{M$(2-\lambda)$}, i.e. the estimates \eqref{A.F_SWE}. 
	
	When $d=1,2$, we can easily apply Hardy-Littelewood-Sobolev inequality (\cite[Theorem 4.3]{LL1997}) for $\lambda<d$ to  bound 
\begin{align*}
	\sup\limits_{x,x'\in\RR^d}& \int_{\RR^{2d}} G^{\upw}_t(x-y)\Lambda(y-y') G^{\upw}_t(x'-y')dydy' \\
	\leq& \sup\limits_{x,x'\in\RR^d} \int_{\RR^{2d}} G_t(x-y)|y-y'|^{-\lambda} G_t(x'-y')dydy' \\
	\leq&  \lk \int_{\RR^{d}}|G^{\upw}_t(y)|^{\frac{2d}{2d-\lambda}} dy\rk^{\frac{2d-\lambda}{d}}\,.
\end{align*}
For $d=1$, we have
\[
 \lk \int_{\RR^{d}}|G^{\upw}_t(y)|^{\frac{2d}{2d-\lambda}} dy\rk^{\frac{2d-\lambda}{d}}\simeq \lk\int_{-t}^t |1/2|^{\frac{2}{2-\lambda}} dt\rk^{2-\lambda} \leq C\cdot t^{2-\lambda}\,.
\]
For $d=2$, we have
\begin{align*}
	\lk \int_{\RR^{d}}|G^{\upw}_t(y)|^{\frac{2d}{2d-\lambda}} dy\rk^{\frac{2d-\lambda}{d}}\simeq&\ \lk\int_{\RR^2} |t^2-x^2|^{-\frac{2}{4-\lambda}}\1_{|x|<t}dx\rk^{\frac{4-\lambda}{2}} \\
	=&\ t^{2-\lambda}\cdot \lk\int_{\RR^2} |1-x^2|^{-\frac{2}{4-\lambda}}\1_{|x|<1}dx\rk^{\frac{4-\lambda}{2}} \\
	=&\  C\cdot t^{2-\lambda}\,,
\end{align*}
where the integral is finite if $\lambda<2$.


Now we  shall  apply the  HLS inequality on sphere (see e.g. \cite[Theorem 4.5]{LL1997})  to show  \eqref{A.F_SWE} for $d=3$ and $\lambda<3$. Denote  by $\SS^3$   the unit sphere in $\RR^3$.  We have  
	\begin{align*}
	\sup\limits_{x,x'\in\RR^3}& \int_{\RR^{3}\times\RR^{3}} G^{\upw}_t(x-y)\Lambda(y-y') G^{\upw}_t(x'-y')dydy' \\
	\leq& \sup\limits_{x,x'\in\RR^3} \int_{\RR^{3}\times\RR^{3}} |y-y'|^{-\lambda} \frac{\sigma_t(x-dy)}{4\pi t}\frac{\sigma_t(x'-dy')}{4\pi t} \\
	\simeq&\ t^{2-\lambda}\cdot \sup\limits_{x,x'\in\RR^3} \int_{\RR^{3}\times\RR^{3}} \1_{x+\SS^{3}}(y) |y-y'|^{-\lambda} \1_{x'+\SS^{3}}(y) \sigma_1(dy)\sigma_1(dy')\\
	\ls&\ t^{2-\lambda}\cdot\sup\limits_{x\in\RR^3}  \lk \int_{\RR^{3}}|\1_{x+\SS^{3}}(y)|^{\frac{6}{6-\lambda}} \sigma_1(dy)\rk^{\frac{6-\lambda}{3}}=C\cdot t^{2-\lambda} \,,
	\end{align*}
	where we have made use of the scaling property of the surface measure $\sigma_t(dy)=t^2\sigma_1(d\tilde{y})$ with $y=t\tilde{y}$ in the third line and  the  HLS inequality  \cite[Theorem 4.5]{LL1997}   on sphere in the last line.   This proves \eqref{A.F_SWE}. 

In regard to the bound  \eqref{A.F_SWE'}, it is easy to see that if $\lambda<2\wedge d$
	\begin{align*}
		\sup_{\eta\in\RR^d} \int_{\RR^d}|\hat{G}^{\upw}_t(\xi)|^2\mu(d\xi-\eta)=&\sup_{\eta\in\RR^d} \int_{\RR^d}\lt|\frac{\sin(t|\xi|)}{\xi} \rt|^2 \mu(d\xi-\eta)  \\
		\leq&\ t^{2-\lambda}\cdot \sup_{\eta\in\RR^d} \int_{\RR^d}\frac{|\xi|^{\lambda-d}}{1+|\xi+t\eta|^2}d\xi  \leq C \cdot t^{2-\lambda}\,.
	\end{align*}
Thus,   we  complete  the proof of Proposition \ref{Prop.UpperM_SWE}. 
\end{proof}

\begin{rmk}
	The properties we obtained  in Proposition \ref{Prop.HSB_Heat} (i) and Proposition \ref{Prop.HSB_Wave} (i) can be also rewritten as the  following \emph{small ball  property} (\emph{B($\sfa$,$\sfb$,$\sfc$)}):     if $y\in B_{\ep}(x)$, then
	\begin{equation}\label{eq.SBP}
		\int_{B_{\ep}(x)}G_{t}(y-z)dz \geq C_1\cdot t^a  \exp\lc -C_2\frac{t^b}{\ep^c}\rc\,,
	\end{equation}
	where $\sfa$, $\sfb$ and $c$ are   parameters depending on the kernel. Obviously, \emph{B($\sfa$,$\sfb$,$\sfc$)} is stronger than \emph{B($\sfa$,$\sfb$)} because \eqref{eq.SBP} holds for all $t>0$ other than $0<t\leq \ep^\beta $.
	
	For example, we have proved  that ($\sfa$,$\sfb$,$\sfc$)=(0,1,2) for the heat kernel, $(\sfa,\sfb,\sfc)=(0,1,\alpha)$ for the $\alpha$-heat kernel and ($\sfa$,$\sfb$,$\sfc$)=(1,1,1) for the wave kernel.
	
	Our effort to take into account 
	 \emph{B($\sfa, \sfb$)} 
	rather than \emph{B(a,b,c)} is mainly stimulated by Proposition \ref{Prop.HSB_Wave} (ii). 
	One should note that when $d=3$, the wave kernel can not satisfy the \emph{B(a,b,c)}. 
	Because the three dimensional wave kernel is a surface measure on the sphere $\partial B_t(0)$, there might be no intersection between the surface measure $G_t^{\upw}(y-dz)$ and the ball $B_{\ep}(x)$ if $t\gg \ep$. Then the lower bound in \eqref{eq.SBP_W.3} might be $0$.
\end{rmk}

\subsection{Fractional temporal  and fractional spatial  equations: space homogeneous case} 
In this section we consider  the following $d$-spatial dimensional stochastic partial differential equation of fractional orders both in time and space variables,
which will be
 called the stochastic fractional diffusion (SFD). The existence, uniqueness, upper moment
bounds have been obtained earlier (e.g. \cite{CHHH2017} , \cite{MN2015} and references therein).
But the sharp lower bounds for any moment has not been known. We shall apply Theorem 
\ref{thm.LMB} to obtain a sharp lower moment bounds for this equation.

This type of  equations takes the following form: 
\begin{equation}\label{eq.FDiffu}
(\text{SFD})\quad \begin{cases}
  \partial_t^{\beta} u(t,x)=-\frac12(-\Delta)^{\alpha/2} 
  u(t,x)+u(t,x)\dot{W}(t,x)\,,  &t>0,\quad x\in\RR^d\,, \\
  \partial_t^{k} u(t,x)|_{t=0}=u_{k}(x)\,,  &0\leq k\leq \lceil \beta \rceil-1\,.
\end{cases}
\end{equation}
As in \cite{CHHH2017, MN2015}, we shall assume 
that  $\beta\in(1/2,2)$ and $\alpha\in(0,2]$. We refer to  \cite{KST2006} for the precise meaning of the fractional derivative in time and the fractional Laplacian. Notice that the SWE coincide with the case $(\alpha,\beta)=(2,2)$ in \eqref{eq.FDiffu} formally. 

In this case, the operator $\sL$ is given by
\[
\sL u(t,x)=\partial_t^{\beta} u(t,x)+\frac12(-\Delta)^{\alpha/2} u(t,x)\,. 
\]
The associated Green's function  can be represented by the Fox $H$-function. 
\begin{equation}\label{eq.GY}
	G^{Y}_{t}(x):=G^{Y;\alpha,\beta,d}_{t}(x)=\frac{t^{\beta-1}}{\pi^{d/2}|x|^d} H_{2,3}^{2,1}\lc\frac{|x|^{\alpha}}{2^{\alpha-1}t^{\beta}}\bigg|\substack{ (1,1),(\beta,\beta)\\ (\frac{d}{2},\frac{\alpha}{2}),(1,1),(1,\frac{\alpha}{2})} \rc\,,
\end{equation}
where $H$  is a Fox H-function (e.g. \cite{KS2004}). 
When $\beta>1$ we also need another Green function 
\begin{equation}\label{eq.GZ}
	G^{Z}_{t}(x):=G^{Z;\alpha,\beta,d}_{t}(x)=\frac{t^{\lceil \beta \rceil-1}}{\pi^{d/2}|x|^d} H_{2,3}^{2,1}\lc\frac{|x|^{\alpha}}{2^{\alpha-1}t^{\beta}}\bigg|\substack{ (1,1),(\lceil\beta\rceil,\beta)\\ (\frac{d}{2},\frac{\alpha}{2}),(1,1),(1,\frac{\alpha}{2})} \rc 
\end{equation}
to represent $I_0(t,x)$,  namely,
\begin{equation}\label{eq.FDiffu_J0}
	I^{\upf}_0(t,x)=\sum_{k=0}^{\lceil\beta\rceil-1}\int_{\RR^d}u_{\lceil\beta\rceil-1-k}(y)\partial_{t}^k G^{Z}_{t}(x-y)dy\,.
\end{equation} 
The Fourier transforms of $G^{Y}_{t}(x)$ and $G^{Z}_{t}(x)$ are given by the following :
\begin{align}
	\cF [G^{Z}_{t}(\cdot)](\xi)=\,&t^{\lceil \beta \rceil-1}E_{\beta,\lceil \beta \rceil}\lc -\frac{t^{\beta}|\xi|^{\alpha}}{2}\rc\,,\nonumber \\
	\cF [G^{Y}_{t}(\cdot)](\xi)=\,&t^{\beta-1}E_{\beta,\beta}\lc -\frac{t^{\beta}|\xi|^{\alpha}}{2}\rc\,,\label{Four.GY}
\end{align}
where $E_{\beta,\beta'}$ is the Mittag-Leffler function (e.g. \cite{KST2006}).  
%

As before, we may assume $u_0=1$ and $u_k=0$ for $k\geq 1$ to simplify the form of moments without loss of generality (also see Remark 3.6 in \cite{CHHH2017}).  We   have $I^{\upf}_0(t,x)=1$ by our particular initial conditions. Whence, we can prove Theorem \ref{thm.UMB} with the notations introduced before.

Positivity  of $G^{Y}_{t}(x)$ (as well as  $G^{Z}_{t}(x)$) 
have been obtained in the following  three cases in  \cite[Theorem 3.1]{CHHH2017}:
\begin{equation*} 
	\begin{cases}
		\hbox{$d=1$, $\beta\in(1,2)$ and $\alpha\in[\beta,2]$}\,;\\
 	\hbox{$d=2,3$, $\beta\in(1,2)$ and $\alpha=2$}\,;\\
  	\hbox{$d\in \bN$, $\beta\in(0,1]$ and $\alpha\in(0,2]$}\,. \\
	\end{cases}
\end{equation*}   
Notice that although $\beta$ is allowed to be smaller than $\frac 12$,
 the existence and uniqueness of solutions to \eqref{eq.FDiffu} {can be  proved}
only  under the conditions $\beta\in(\frac 12,2)$ and $\alpha\in(0,2]$. 
 Therefore, we will replace last condition   by
 \[
  d\in \bN\,, \beta\in(\frac 12,1]\,, \, \text{and}\, \alpha\in(0,2]\,.
 \]
This means that  we will assume that
$(\alpha,\beta,d)$ satisfies one of the following three conditions:
\begin{equation}\label{FD_Conds}
	\begin{cases}
		(a)~\beta\in(\frac 12,1] \text{ and } \alpha\in(0,2], &d\in \bN\,;\\
		(b)~\beta\in(1,2) \text{ and } \alpha\in(0,2], &d=2,3\,;\\
		(c)~\beta\in(1,2) \text{ and } \alpha\in[\beta,2], &d=1\,.
	\end{cases}
\end{equation}
As we indicated 
 above  the assumption \ref{A1} is
 met under the above  parameter range of \eqref{FD_Conds}. 
 In the remaining part of this subsection, we shall prove 
   \ref{A2} and \ref{A3} for the Green's function  $G^{Y}_{t}$.

\begin{prop}[\textbf{Small Ball Nondegeneracy Property and Lower Moments for SFD}]\label{Prop.HSB_FD} 
	For the kernel $G^{Y}_{t}(x)$ defined in \eqref{eq.GY}, 
	the \emph{small ball nondegeneracy 
	property}  \emph{B($\beta-1$,$\frac{\alpha}{\beta}$)} holds 
	for the parameter ranges given in  \eqref{FD_Conds}.
	More precisely, there exist a strictly  
	positive constant $C$ independent of $t$, $\ep$ and $y$ such that
\begin{equation}\label{eq.SB_FD}
	\inf_{y\in B_{\ep}(x)} 
	\int_{B_{\ep}(x)}G^{Y}_{t}(y-z)dz\geq C\cdot t^{\beta-1}
\end{equation}
for any $0<t\leq \ep^{\frac{\alpha}{\beta}}$.

As a result, if $\gamma(\cdot)$ (with $\gamma=2-2H$) 
and $\Lambda(\cdot)$ satisfy  the same conditions as in 
Theorem \ref{thm.LMB},  the lower $p$-th ($p\geq 2$) 
moment bounds   hold
	\[
	 \EE[|u^{\upf}(t,x)|^p]\geq c_1 \exp\lc c_2\cdot t^{\frac{\alpha(2\beta+2H-2)-\beta\lambda}{2\alpha\beta-\alpha-\beta\lambda}}\cdot p^{\frac{\beta(2\alpha-\lambda)}{2\alpha\beta-\alpha-\beta\lambda}} \rc 
	\]
	for some constants $c_1$ and $c_2$ 
independent of $t$, $p$ and $x$. 
\end{prop}


\begin{proof}
We divide the proof into three steps to deal with three cases in \eqref{FD_Conds} seperately. 

\noindent \textbf{Step 1: case (a).} The special case  $\beta=1$ 
was treated in \eqref{eq.BP_H.2}, so we can assume  $\beta\in(1/2,1)$. 
By the convolution property of  \cite{CHHH2017}, we get a subordination law
for the Green's function: 
\begin{align}\label{eq.GY_Rep.a}
	G^{Y}_{t}(x)=\,&\frac{t^{\beta-1}}{\pi^{d/2}|x|^d} H_{2,3}^{2,1}\lc\frac{|x|^{\alpha}}{2^{\alpha-1}t^{\beta}}\bigg|\substack{ (1,1),(\beta,\beta)\\ (\frac{d}{2},\frac{\alpha}{2}),(1,1),(1,\frac{\alpha}{2})} \rc\nonumber \\
	=\,&\frac{\beta t^{\beta-1}}{\pi^{d/2} |x|^{d}}\int_0^{\infty} H_{1,2}^{1,1}\lc\frac{|x|^{\alpha}s^{\beta}}{2^{\alpha-1}}\bigg|\substack{ (1,1)\\ (\frac{d}{2},\frac{\alpha}{2}),(1,\frac{\alpha}{2})} \rc H_{1,1}^{1,0}\lc (ts)^{-\beta}\bigg|\substack{ (\beta,\beta)\\ (1,1)} \rc \frac{ds}{s}\,.
\end{align}
When  $y,z\in B_{\ep}(x)$,  $t\leq \ep^{\frac{\alpha}{\beta}}$ and when $\ep$ is 
small enough we have 
\begin{align}\label{GY.Est_a1}
	\int_{B_{\ep}(x)}G&^{Y}_{t}(y-z)dz \nonumber\\
	&\simeq \int_{B_{\ep}(x)} \frac{\beta t^{\beta-1}}{|y-z|^{d}}\int_0^{\infty} H_{1,2}^{1,1}\lc\frac{|y-z|^{\alpha}s^{\beta}}{2^{\alpha-1}}\bigg|\substack{ (1,1)\\ (\frac{d}{2},\frac{\alpha}{2}),(1,\frac{\alpha}{2})} \rc \nonumber\\
	&\qquad\qquad\qquad\qquad\qquad\ \times H_{1,1}^{1,0}\lc (ts)^{-\beta}\bigg|\substack{ (\beta,\beta)\\ (1,1)} \rc \frac{ds}{s} dz\nonumber\\
	&\simeq \int_{B_{\ep}(x)} \frac{\beta t^{\beta-1}}{|y-z|^{d}}\int_0^{\infty} H_{1,2}^{1,1}\lc\frac{|y-z|^{\alpha}s^{\beta}}{2^{\alpha-1}t^{\beta}}\bigg|\substack{ (1,1)\\ (\frac{d}{2},\frac{\alpha}{2}),(1,\frac{\alpha}{2})} \rc \nonumber\\
	&\qquad\qquad\qquad\qquad\qquad\quad\ \times H_{1,1}^{1,0}\lc s^{-\beta}\bigg|\substack{ (\beta,\beta)\\ (1,1)} \rc \frac{ds}{s} dz\,.
\end{align} 
Notice that the second $H$-function is nonnegative by Lemma 4.5 in\cite{CHHH2017}. Moreover, recall that  the characteristic function and the density of a centered, $d$-dimensional spherically symmetric $\alpha$-stable random variable are given, respectively,  by 
\begin{equation}\label{eq.al_Dens}
	f_{\alpha,d}(\xi)=\exp(-|\xi|^{\alpha})\,, \qquad\xi\in\RR^d\,,
\end{equation}
and
\begin{equation}\label{eq.al_Char}
	\rho_{\alpha,d}(x)=\frac{1}{(\sqrt{\pi}|x|)^d}H_{1,2}^{1,1}\lc\frac{|x|^{\alpha}}{2^{\alpha}}\bigg|\substack{ (1,1)\\ (\frac{d}{2},\frac{\alpha}{2}),(1,\frac{\alpha}{2})} \rc\,, \qquad x\in\RR^d\,.
\end{equation}
This means that  the first Fox H-function is related to   the spherically symmetric $\alpha$-stable distribution (see also  \cite[Theorem 3.3]{CHHH2017} for more details).
Therefore, one can apply the Pollard's formula in \cite{CHW2018} together with \eqref{eq.al_Dens} and \eqref{eq.al_Char} to find  
\begin{align*}
	\frac{1}{|y-z|^{d}}  & H_{1,2}^{1,1}\lc\frac{|y-z|^{\alpha}s^{\beta}}{2^{\alpha-1}t^{\beta}}
 \bigg|\substack{ (1,1)\\ (\frac{d}{2},\frac{\alpha}{2}),(1,\frac{\alpha}{2})} \rc \simeq \ G_{(t/s)^{\beta}}^{\uph,\alpha}(y-z) \\
&\qquad\qquad 	\simeq \ \lc \frac{t}{s}\rc^{-\frac{\beta d}{\alpha}}\wedge \frac{(t/s)^{\beta}}{|y-z|^{d+\alpha}} \nonumber \\
	 & \qquad\qquad \gtrsim \ \lc \frac{t}{s}\rc^{-\frac{\beta d}{\alpha}} \exp\lc -C_{\alpha,d}\cdot\frac{|y-z|^{\alpha}}{(t/s)^{\beta}} \rc \,,
\end{align*}
where $G_t^{\uph,\alpha}(x)$ is the $\alpha$-heat kernel associated to \eqref{eq.SHE_al}.
Whence, by Proposition \ref{Prop.HSB_Heat} (i),  and  \eqref{GY.Est_a1} we   get  
\begin{align}\label{GY.Est_a2}
	\int_{B_{\ep}(x)}G&^{Y}_{t}(y-z)dz \nonumber \\
	&\gtrsim t^{\beta-1}\int_0^{\infty} \exp\lc-c\cdot\frac{(t/s)^{\beta}}{\ep^{\alpha}} \rc \times H_{1,1}^{1,0}\lc s^{-\beta}\bigg|\substack{ (\beta,\beta)\\ (1,1)} \rc \frac{ds}{s} \nonumber\\
	&\gtrsim t^{\beta-1}\int_0^{\infty} \exp\lc-c\cdot s \rc \times H_{1,1}^{1,0}\lc s \bigg|\substack{ (\beta,\beta)\\ (1,1)} \rc \frac{ds}{s} \,,
\end{align} 
if $y,z\in B_{\ep}(x)$ and $t<\ep^{\alpha/\beta}$. 

Next, we need to  analyze $H_{1,1}^{1,0}\lc s \bigg|\substack{ (\beta,\beta)\\ (1,1)} \rc$.
We only need to consider its asymptotics  
for $s$ near  to $0$ and near $\infty$. 
We shall use the results in  the Appendix of \cite{CHHH2017} 
replacing the notations there by   $\Delta=\beta_1-\alpha_1=1-\beta$, $a^*=\beta_1-\alpha_1=1-\beta$, 
$\delta=\beta^{-\beta}$ and $\mu=1-\beta$. 
Let us recall the asymptotic expansion   for the Fox H-function
(e.g. \cite[(A10)]{CHHH2017}): 
\begin{equation}\label{Asymp.HExp}
	H_{p,q}^{m,n}\lc s\, \bigg|\substack{ (a_{i},\alpha_{i})_{1,p}\\ (b_{j},\beta_{j})_{1,q}}\rc =\, \sum_{j=1}^{m}\sum_{l=0}^{\infty}  h_{jl}^{*}\cdot s^{\frac{b_j+l}{\beta_j}}\,. 	
\end{equation}
Thus,  when $s \to 0$  we have 
\begin{align}\label{Asymp_a.H1}
	H_{1,1}^{1,0}\lc s\, \bigg|\substack{ (a_1,\alpha_1)\\ (b_1,\beta_1)} \rc= H_{1,1}^{1,0}\lc s \bigg|\substack{ (\beta,\beta)\\ (1,1)} \rc 
	=\sum_{l=0}^{\infty}h_{l}^{*}\cdot s^{l+1} \,,
\end{align}
since $ m=1$ and $(b_1,\beta_1)=(1,1)$,  $h_{l}^{*}$ 
is given by  (e.g. \cite[(A.12)]{CHHH2017})
\begin{align*}
	h_{l}^{*}=\,&\frac{(-1)^l}{l! \beta_1}\cdot \frac{1}{\Gamma\lc a_1-[b_1+l]\frac{\alpha_1}{\beta_1}\rc}
	=\,\frac{(-1)^l }{l! }\cdot \frac{1}{\Gamma\lc-\beta l \rc}\,.
\end{align*}
Therefore, one can easily see that $h_{0}^{*}=0$, $h_{1}^{*}=-1/\Gamma(-\beta)>0$, and
\[
H_{1,1}^{1,0}\lc s \bigg|\substack{ (\beta,\beta)\\ (1,1)} \rc 
	=\sum_{l=0}^{\infty}h_{l}^{*}\cdot s^{l+1}\simeq h_{1}^{*}\cdot s\,,\qquad |s|\simeq 0\,.
\]
When $s$ goes to infinity, by     \cite[Corollary 1.10.2]{KS2004}, we have the following asymptotic:
\begin{eqnarray}\label{Asymp_a.H2}
	H_{1,1}^{1,0}\left(s \bigg|\substack{ (\beta,\beta)\\ (1,1)}
	\right)
	&=&O\lc s^{[3/2-\beta]/(1-\beta)}\exp\lk-C_\beta s^{1/(1-\beta)} \rk \rc\,,\qquad s\to \infty\,,
	\nonumber\\
	&\gtrsim & \exp(-C_{\beta}s^{1/(1-\beta)})\,,  
\end{eqnarray} 
where
$C_\beta=(1-\beta) \beta^{\beta/(1-\beta)}$.
Whence we can observe that the integral in \eqref{GY.Est_a2} is finite. So, we have for  some constant $C_{\beta}>0$
\begin{align*}
	\int_{B_{\ep}(x)}G&^{Y}_{t}(y-z)dz  \gtrsim C_{\beta}\cdot t^{\beta-1}\,.
\end{align*}
As a result, we have proved the small ball nondegeneracy 
property \emph{B($\beta-1$,$\frac{\alpha}{\beta}$)} for the \textbf{case (a)}.

\noindent\textbf{Step 2: case (b).} In this case  $d=2$ or  $d=3$, $\beta\in(1,2)$ and $\alpha=2$.  By equations (43) and (85) in \cite{P2009}, we have for $\beta\in[1,2)$
\begin{align*}
	&G^{Y}_{t}(x)=\,\Gamma_{\beta,d}(t,x)\,, 
\end{align*}
where 
\begin{align}
	\Gamma_{\beta,2}(t,x)=\,& \frac{C\cdot t^{-1}}{\Gamma(1/2)} \int_1^{\infty} \phi(-\beta/2,0,-|x|t^{-\frac{\beta}{2}}\tau)(\tau^2-1)^{-1/2}d\tau\,,\label{eq.GY_Rep.b2}\\
	\Gamma_{\beta,3}(t,x)=\,&C t^{-\frac{\beta}{2}-1}\int_{1}^{\infty}\phi(-\beta/2,-\beta/2;-|x|t^{-\beta/2}) d\tau\, .\label{eq.GY_Rep.b3}
\end{align}
Here $\phi(a,b,c)$ is the Wright function. 

Let us check the \emph{small ball nondegeneracy property} \emph{B($\beta-1$,$\frac{2}{\beta}$)} for $d=2$   first. If $y,z\in B_{\ep}(x)$ and $t\leq \ep^{2/\beta}$, by the representation \eqref{eq.GY_Rep.b2}
\begin{align*}
	\int_{B_{\ep}(x)}G&^{Y}_{t}(y-z)dz= \int_{B_{\ep}(x)}\Gamma_{\beta,2}(t,y-z)dz \\
	&\simeq\,t^{-1}\int_{B_{\ep}(x)} \int_1^{\infty} \phi(-\beta/2,0,-|y-z|t^{-\frac{\beta}{2}}\tau)(\tau^2-1)^{-1/2}d\tau dz \\
	&\simeq\,t^{-1}\int_{B_{\ep}(x)}\int_{0}^{\infty} \phi(-\beta/2,0,-\tau) t^{\frac{\beta}{2}}\cdot \frac{t^{\frac{\beta}{2}}\cdot \1_{\{|y-z|\leq \tau t^{\beta/2}\}}}{\sqrt{t^{\beta}\tau^2-|y-z|^2}}  d\tau dz \\
	&\gtrsim\,t^{-1}\int_0^{\infty} \phi(-\beta/2,0,-\tau)\cdot t^{\beta}\tau\cdot  \exp\lc-\frac{t^{\beta/2}\tau}{\ep} \rc d\tau\,,
\end{align*}
where the last inequality is derived analogously to the argument
 used in \eqref{eq.SBP_W2} for the wave kernel when $d=2$ and the fact that $\phi(-\beta/2,0,-\tau)$ is positive (see \cite[Section 2]{P2009}). Then since $t^{\beta/2}\leq \ep$ and $\exp(-t^{\beta/2}\tau/\ep)\geq \exp(-\tau)$, we obtain  by the relation between  the Wright function and 
 the Fox $H$-function 
\begin{align}\label{GY.Est_b1}
	\int_{B_{\ep}(x)}G&^{Y}_{t}(y-z)dz = \int_{B_{\ep}(x)}\Gamma_{\beta,2}(t,y-z)dz \nonumber \\
	\gtrsim&\,t^{\beta-1}\int_0^{\infty} \phi(-\beta/2,0,-\tau)\cdot \tau  \exp\lc-\tau \rc d\tau \nonumber \\
	\simeq&\, t^{\beta-1}\int_0^{\infty} H_{1,1}^{1,0}\lc \tau \bigg|\substack{ (0,\beta/2)\\ (0,1)} \rc\cdot \tau  \exp\lc-\tau \rc d\tau \simeq C_{\beta}\cdot t^{\beta-1} \,,
\end{align}
where the integral in the last equality of \eqref{GY.Est_b1} is finite  by the similar asymptotic analysis of $H_{1,1}^{1,0}$ as in \textbf{case (a)}.
Thus, we   proved   \emph{B($\beta-1$,$\frac{2}{\beta}$)} for $d=2$.


Next, let us check the \emph{small ball nondegeneracy property} \emph{B($\beta-1$,$\frac{2}{\beta}$)} for $d=3$. We have by the equation \eqref{eq.GY_Rep.b3}
\begin{align}\label{GY.Est_b2}
	\int_{B_{\ep}(x)}G&^{Y}_{t}(y-z)dz= \int_{B_{\ep}(x)}\Gamma_{\beta,3}(t,y-z)dz\nonumber \\
	&\simeq\,\int_{B_{\ep}(x)}  t^{-\frac{\beta}{2}-1}\int_{1}^{\infty}\phi(-\beta/2,-\beta/2;-|y-z|t^{-\beta/2}) d\tau dz \nonumber\\
	&\simeq\,t^{-\beta-1}\int_{B_{\ep}(x)} \frac{1}{|y-z|} \int_0^{\infty}\phi(-\beta/2,-\beta/2;-\tau) \1_{\{|y-z|\leq t^{\beta/2}\tau \}} d\tau dz\,.
\end{align}
Now  we can apply the same three dimensional spherical coordinate 
transformation as   in the proof of Proposition \ref{Prop.HSB_Wave}
(now for $d=3$). Assuming  $x=0$, 
  the integral with respect to $z$ in \eqref{GY.Est_b2} becomes
\begin{align*}
	\int_{B_{\ep}(0)} \frac{1}{|y-z|}&\1_{\{|y-z|\leq t^{\beta/2}\tau \}} dz 
	\simeq\,\int_{B_{\tau t^{\beta/2}}(0)}\frac{\1_{B_{\ep}(0)}(y-z)}{|z|} dz \\
	\simeq&\, \tau^2 t^{\beta} \int_{0}^{\tau t^{\beta/2}}\int_{0}^{2\pi}\int_{0}^{\pi} r\cdot \1_{B_{\ep}}(y-\Psi(\theta,\phi))|\sin(\phi)| d\phi d\theta dr \\
	\gtrsim&\, \tau^4 t^{2\beta} \cdot \int_{0}^{2\pi}\int_{0}^{\pi/3} |\sin(\phi)| d\phi d\theta \simeq \tau^4 t^{2\beta}\,.
\end{align*}
Thus, plugging it back to \eqref{GY.Est_b2}, we get
\begin{align*}
	\int_{B_{\ep}(x)}G&^{Y}_{t}(y-z)dz \gtrsim t^{\beta-1} \int_0^{\infty}\phi(-\beta/2,-\beta/2;-\tau)\cdot \tau^4 d\tau \simeq t^{\beta-1}\,,
\end{align*}
where the last equality follows from  the  asymptotic
behavior  of the Wright function. Hence we 
complete  the proof of  the proposition in \textbf{case (b)}.

\noindent\textbf{Step 3: case (c).} We have $d=1$, $\beta\in(1,2)$ and $\alpha\in[\beta,2]$. By Remark 3.2 (3) and convolution property Theorem 1.8 in  \cite{CHHH2017}, the  Fox H-function 
admits an alternative representation:
\begin{align}\label{eq.GY_Rep.c}
	G^{Y}_{t}(x)=\,&\frac{{t^{\beta-1}}}{|x|} H_{3,3}^{2,1}\lc\frac{|x|^{\alpha}}{t^{\beta}}\bigg|\substack{ (1,1),({\beta},\beta),(1,\frac{\alpha}{2})\\ (1,1),(1,\alpha),(1,\frac{\alpha}{2})} \rc\nonumber \\
	=\,&\frac{{\beta t^{\beta-1}}}{|x|}\int_0^{\infty} H_{2,2}^{1,1}\lc |x|^{\alpha}s^{\beta}\bigg|\substack{ (1,1),(1,\frac{\alpha}{2})\\ (1,1),(1,\frac{\alpha}{2})} \rc H_{1,1}^{1,0}\lc (ts)^{-\beta}\bigg|\substack{ ({\beta},\beta)\\ (1,\alpha)} \rc \frac{ds}{s} \nonumber \\
	=\,&\frac{{\beta t^{\beta-1}}}{|x|}\int_0^{\infty} H_{2,2}^{1,1}\lc \frac{|x|^{\alpha}s^{\beta}}{t^{\beta}} \bigg|\substack{ (1,1),(1,\frac{\alpha}{2})\\ (1,1),(1,\frac{\alpha}{2})} \rc H_{1,1}^{1,0}\lc s^{-\beta}\bigg|\substack{ ({\beta},\beta)\\ (1,\alpha)} \rc \frac{ds}{s}\,.
\end{align}
(The representation is well defined since $\Delta_1=\sum_{j=1}^{2}\beta_j-\sum_{j=1}^{2}\alpha_j=0$, $a^*_1=\alpha_1-\alpha_2+\beta_1-\beta_2=2-\alpha$, $\delta_2=\lc\frac{\alpha}{2}\rc^{\alpha/2}\lc\frac{\alpha}{2}\rc^{-\alpha/2}=1$, $\mu_1=2-2=0$; $\Delta_2=\beta_1-\alpha_1=\alpha-\beta$, $a^*_2=\beta_1-\alpha_1=\alpha-\beta$, $\delta_2=\beta^{-\beta}$ and $\mu_2=1-\beta$.) Note that the second Fox $H$-function is nonnegative combining \cite[Property 2.4]{KS2004} with \cite[Lemma 4.5]{CHHH2017}.  By   \cite[(4.38)]{MLP2001}, 
  the first Fox $H$-function can be identified as the Green function of neutral-fractional diffusion, namely, 
\begin{align*}
	\frac{1}{|x|}H_{2,2}^{1,1}\lc |x|^{\alpha} \bigg|\substack{ (1,1),(1,\frac{\alpha}{2})\\ (1,1),(1,\frac{\alpha}{2})} \rc =\,&N_{\alpha}^{0}(|x|)=K_{\alpha,\alpha}^{0}(|x|)\\
	=\,&\frac{1}{\pi}\frac{|x|^{\alpha-1}\sin[\alpha\pi/2]}{1+2|x|^{\alpha}\cos[\alpha\pi/2]+|x|^{2\alpha}}\,.
\end{align*}
From \eqref{eq.GY_Rep.c}  it then follows 
\begin{align*}
	G^{Y}_{t}(x)=\, \beta t^{\beta-1} \int_0^{\infty} \lc\frac{s}{t}\rc^{\beta/\alpha} N_{\alpha}^{0}\lc|x|(s/t)^{\beta/\alpha}\rc H_{1,1}^{1,0}\lc s^{-\beta}\bigg|\substack{ ({\beta},\beta)\\ (1,\alpha)} \rc \frac{ds}{s}\,.
\end{align*}
Thus,  we have (without loss of generality we can set   $x=0$ in the following), 
\begin{align}\label{GY.Est_c}
	&\int_{B_{\ep}(x)}G^{Y}_{t}(y-z)dz \nonumber\\
	=\,&\int_{B_{\ep}(0)} \beta t^{\beta-1} \int_0^{\infty} \lc\frac{s}{t}\rc^{\beta/\alpha} N_{\alpha}^{0}\lc|y-z|(s/t)^{\beta/\alpha}\rc H_{1,1}^{1,0}\lc s^{-\beta}\bigg|\substack{ ({\beta},\beta)\\ (1,\alpha)} \rc \frac{ds}{s} dz \nonumber\\
	\gtrsim\,&\sin\lk\frac{\alpha\pi}{2}\rk t^{\beta-1} \int_0^{\infty}\int_{B_{\ep}(y)} \lc\frac{s}{t}\rc^{\beta/\alpha} \frac{[|z|(s/t)^{\beta/\alpha}]^{\alpha-1}}{[|z|(s/t)^{\beta/\alpha}]^{2\alpha}+1} dz \cdot H_{1,1}^{1,0}\lc s^{-\beta}\bigg|\substack{ ({\beta},\beta)\\ (1,\alpha)} \rc\frac{ds}{s} \nonumber\\
	\gtrsim\,& \sin\lk\frac{\alpha\pi}{2}\rk t^{\beta-1} \int_0^{\infty}\int_{0}^{(s/t)^{\beta/\alpha}\ep} \frac{z^{\alpha-1}}{z^{2\alpha}+1}dz \cdot H_{1,1}^{1,0}\lc s^{-\beta}\bigg|\substack{ ({\beta},\beta)\\ (1,\alpha)} \rc\frac{ds}{s} \nonumber\\
	\simeq\,& \sin\lk\frac{\alpha\pi}{2}\rk t^{\beta-1} \int_0^{\infty} \arctan\lk \frac{s^\beta \ep^{\alpha}}{t^{\beta}}\rk \cdot H_{1,1}^{1,0}\lc s^{-\beta}\bigg|\substack{ ({\beta},\beta)\\ (1,\alpha)} \rc\frac{ds}{s}  \nonumber\\
	\gtrsim\,& \sin\lk\frac{\alpha\pi}{2}\rk t^{\beta-1} \int_0^{\infty} \arctan\lk s^\beta \rk  \cdot H_{1,1}^{1,0}\lc s^{-\beta}\bigg|\substack{ ({\beta},\beta)\\ (1,\alpha)} \rc\frac{ds}{s}\,
\end{align}
for $y,z\in B_{\ep}(x)$, and $t\leq \ep^{\frac{\alpha}{\beta}}$.

Next, we need to take care of the asymptotics 
of $H_{1,1}^{1,0}\lc s^{-\beta}\bigg|\substack{ (\beta,\beta)\\ (1,\alpha)} \rc$ 
(with the notations $\Delta=\beta_1-\alpha_1=\alpha-\beta$, 
$a^*=\beta_1-\alpha_1=\alpha-\beta$, $\delta=\beta^{-\beta}$ 
and $\mu=1-\beta$) when $s$ goes to infinity. 
Similar to  \eqref{Asymp_a.H1} in \textbf{case (a)}, we   find that
\begin{equation*}
	H_{1,1}^{1,0}\lc s^{-\beta}\bigg|\substack{ ({\beta},\beta)\\ (1,\alpha)} \rc\,=\sum_{l=0}^{\infty}h_{l}^{*}\cdot s^{-\beta(l+1)/\alpha} \simeq\, h_1^{*}s^{-2\beta/\alpha}\, \qquad \text{as }s\to\infty\,,
\end{equation*}
with $h_l^{*}=\frac{(-1)^l }{\alpha\cdot l!}
\cdot \frac{1}{\Gamma\lc-\beta l \rc}$ and $ h_1^{*}>0$.
 When $s\to 0$,  similar to   \eqref{Asymp_a.H2}, 
 we have the following asymptotic estimate 
\[
 H_{1,1}^{1,0}(s^{-\beta}|\cdots)=O\lc s^{-\beta[3/2-\beta]/(\beta-\beta)}\exp\lk-C_{\alpha,\beta}\cdot s^{-\beta/(\alpha-\beta)} \rk \rc\,,\qquad s\to  0 
\]
   for some constant $C_{\alpha,\beta}>0$.
   
Finally, we obtain from \eqref{GY.Est_c} and the asymptotics
\begin{align*}
	\int_{B_{\ep}(x)}G&^{Y}_{t}(y-z)dz \\
	\gtrsim\,& \sin\lk\frac{\alpha\pi}{2}\rk t^{\beta-1} \int_0^{\infty} \arctan\lk s^\beta \rk  \cdot H_{1,1}^{1,0}\lc s^{-\beta}\bigg|\substack{ ({\beta},\beta)\\ (1,\alpha)} \rc\frac{ds}{s}\gtrsim\, C_{\alpha,\beta}\cdot t^{\beta-1}\,,
\end{align*}
for some constant $C_{\alpha,\beta}>0$. Thus, we complete 
 the proof of the small ball nondegeneracy property \emph{B($\beta-1$,$\frac{\alpha}{\beta}$)} for \textbf{case (c)}. 
\end{proof}

\begin{prop}[\textbf{HLS mass Property  and  Upper Moments for SFD}]\label{Prop.UpperM_SFD}
	Assume that $\gamma(\cdot)$ (with $\gamma-2-2H$) and $\Lambda(\cdot)$  
	satisfy the same conditions of  Theorem \ref{thm.UMB} (under the condition
	$\lambda<\min(2\alpha-\alpha/\beta,d)$) or Theorem \ref{thm.UMB2} (under the condition
	 $\lambda<\min(\alpha,d)$). 
When the parameters are in the range   given by \eqref{FD_Conds} the Green's function  $G^{Y}_{t}(x)$ satisfies the  
	\ref{A3} or \ref{A3'} with \emph{M$(2(\beta-1)-\frac{\beta\lambda}{\alpha})$}.
	  In other words, there exist  strict positive constants $C_1$ and $C_2$ 
	  independent of  $t$ and $x$ such that
	\begin{equation}\label{A.F_SFD}
	\sup\limits_{x,x'\in\RR^d} \int_{\RR^{2d}} G^{Y}_t(x-y)\Lambda(y-y') G^{Y}_t(x'-y')dydy' \leq C\cdot t^{2(\beta-1)-\frac{\beta\lambda}{\alpha}}\,,
	\end{equation}
and furthermore, 
denoting  $\mu(d\xi)=\hat{V}(\xi) d\xi$
	\begin{equation}\label{A.F_SFD'}
	\sup_{\eta\in\RR^d} \int_{\RR^d}|\hat{G}^{Y}_t(\xi-\eta)|^2 |\mu|(d\xi) \leq C_3 \cdot t^{2(\beta-1)-\frac{\beta\lambda}{\alpha}}\,.
	\end{equation}
	
	Consequently,  we have  the upper $p$-th ($p\geq 2$) moment bounds 
	for the solution $u^{\upf}(t,x)$. Namely,  there are positive 
	 constants $C_1$ and $C_2$ independent of $t$, $p$ and $x$ satisfying 
		 $$\EE[|u^{\upf}(t,x)|^p]\leq C_1\cdot \exp\lc C_2\cdot t^{\frac{\alpha(2\beta+2H-2)-\beta\lambda}{2\alpha\beta-\alpha-\beta\lambda}}\cdot p^{\frac{\beta(2\alpha-\lambda)}{2\alpha\beta-\alpha-\beta\lambda}} \rc\,.$$
\end{prop}
\begin{proof}
	We need to show \emph{M$(2(\beta-1)-\frac{\beta\lambda}{\alpha})$}   under conditions \eqref{FD_Conds} and $\lambda<\min(2\alpha-\alpha/\beta,d)$, i.e.  the estimates \eqref{A.F_SFD}. This gives  the upper bound accordingly.  This has been proved  in \cite[Theorem 3.14 and Lemma 7.3]{CHHH2017}. For the sake of completeness, we give some details here.  Applying Hardy-Littelewood-Sobolev inequality (\cite[Theorem 4.3]{LL1997}), we can find 
\begin{align*}
	\sup\limits_{x,x'\in\RR^d}& \int_{\RR^{2d}} G^{Y}_t(x-y)\Lambda(y-y') G^{Y}_t(x'-y')dydy' \\
	\leq& \sup\limits_{x,x'\in\RR^d} \int_{\RR^{2d}} G^{Y}_t(x-y)|y-y'|^{-\lambda} G^{Y}_t(x'-y')dydy' \\
	\leq&  \lk \int_{\RR^{d}}|G^{Y}_t(y)|^{\frac{2d}{2d-\lambda}} dy\rk^{\frac{2d-\lambda}{d}} \simeq \lk \int_{\RR^{d}}\lt| \frac{t^{\beta-1}}{|y|^d} H_{2,3}^{2,1}\lc\frac{|y|^{\alpha}}{2^{\alpha-1}t^{\beta}}\bigg|\substack{ --\\ ---} \rc \rt|^{\frac{2d}{2d-\lambda}} dy\rk^{\frac{2d-\lambda}{d}} \\
	\simeq&\ t^{2(\beta-1)-\frac{\beta\lambda}{\alpha}}\cdot \lk \int_{\RR^{d}}\lt|\frac{1}{|y|^d} H_{2,3}^{2,1}\lc |y|^{\alpha}\bigg|\substack{ --\\ ---} \rc \rt|^{\frac{2d}{2d-\lambda}} dy\rk^{\frac{2d-\lambda}{d}} \leq C\cdot t^{2(\beta-1)-\frac{\beta\lambda}{\alpha}}\,,
\end{align*}
where we have employed change of variable $y\to t^{\beta/\alpha}\cdot y$ and the estimate of H-function $H_{2,3}^{2,1}(y)$ obtained in \cite[Lemma 7.1]{CHHH2017}.
	
	
	Next, we need to prove the inequality \eqref{A.F_SFD'} under conditions \eqref{FD_Conds} and $\lambda<\min(\alpha,d)$. Let us recall some useful estimates for the Mittag-Leffler function $E_{\beta,\beta}(-|z|^{\beta})=\sum_{k=0}^{\infty}\frac{(-|z|)^k}{\Gamma(\beta(k+1))}$ (see \cite{GLL2002} or \cite{WZO2018} for example): when $z\to \infty$, 
	\begin{align*}
		|E_{\beta,\beta}(-|z|)| \,\ls\, |z|^{-1}+|z|^{-2}\,.
	\end{align*} 	
On the other hand  the Mittag-Leffler function $E_{\beta,\beta}(-|z|^{\beta})$ is bounded when $|z|\simeq 0$ for $\beta\in (0,2)$. 
	Therefore, the following inequality holds
	\begin{align*}
		| E_{\beta,\beta}(-|z|^{\alpha})| \ \ls\ 1\wedge |z|^{-\alpha} \ \ls\ {\frac{1}{1+|z|^{\alpha}}}\,.
	\end{align*}
	Using the equation \eqref{Four.GY} and the assumptions on $\Lambda(\cdot)$, we have
	\begin{align*}
		\sup\limits_{x\in\RR^d} \int_{\RR^d} G_t^{Y}(x-y)\Lambda(y)dy \ \ls & \ \sup\limits_{x\in\RR^d} \int_{\RR^d} G_t^{Y}(x-y)|y|^{-\lambda}dy \\
		\ \simeq & \ \sup\limits_{x\in\RR^d} \int_{\RR^d} \hat{G}^{Y}_t(\xi)\cdot e^{\iota x\cdot \xi} |\xi|^{\lambda-d}d\xi \\
		\ \ls & \ t^{\beta-1}\cdot\int_{\RR^d} \lt|E_{\beta,\beta}\lc-\frac{t^\beta|\xi|^{\alpha}}{2}\rc\rt|\cdot |\xi|^{\lambda-d}d\xi \\
		\ \ls & \ t^{(\beta-1)-\frac{\beta\lambda}{\alpha}}\cdot\int_{\RR^d} \lt|E_{\beta,\beta}\lc-|\xi|^{\alpha}\rc\rt|\cdot |\xi|^{\lambda-d}d\xi\,.
	\end{align*}
	And the integral is well defined since
	\begin{align*}
		\int_{\RR^d} \lt|E_{\beta,\beta}\lc-|\xi|^{\alpha}\rc\rt|\cdot |\xi|^{\lambda-d}d\xi \ \ls & \ {\int_{\RR^d} \lk 1+\lt|\xi\rt|^{\alpha}\rk^{-1} \cdot |\xi|^{\lambda-d}d\xi <\infty}\,,
	\end{align*}
	under {the assumption $\lambda<\min(\alpha,d)$}. 
	Thus, we complete the proof. 
\end{proof}

\bibliography{Intermittency_properties_in_SPDEs_Sep8}
\bibliographystyle{plain}

\end{document}